\documentclass[a4j,12pt, reqno]{amsart} 
\usepackage{amsmath,amssymb,epic,eepic,amscd,epsfig, amsthm}
\usepackage{comment} 
\usepackage{mathrsfs}
\usepackage[all,cmtip]{xy}
\usepackage{graphicx}
\usepackage{here}

\makeatletter
\def\@settitle{\begin{center}%
  \baselineskip14\p@\relax
  \normalfont\Large\bfseries
  \@title
  \ifx\@subtitle\@empty\else
     \\[1ex] 
     \normalsize\mdseries\@subtitle
  \fi
 \ifx\@didication\@empty\else
     \\[2ex] 
     \small     
     \it\@dedication
  \fi
  \end{center}%
}
\def\subtitle#1{\gdef\@subtitle{#1}}
\def\@subtitle{}
\def\dedication#1{\gdef\@dedication{#1}}
\def\@dedication{}

\renewcommand{\section}{\@startsection
{section}{1}{0mm}{5mm}{2mm}{\raggedright\bfseries}}

\makeatother

\newtheorem{theorem}{Theorem}  
\newtheorem{Theorem}[theorem]{Theorem}
\newtheorem{Lemma}[theorem]{Lemma}
\newtheorem{Corollary}[theorem]{Corollary}
\newtheorem{Proposition}[theorem]{Proposition}
\theoremstyle{definition}
\newtheorem{Definition}[theorem]{Definition}

\newtheorem{Remark}[theorem]{Remark}

\newtheorem{Question}[theorem]{Question}
\newcommand*{\longhookrightarrow}{\ensuremath{\lhook\joinrel\relbar\joinrel\rightarrow}}
\newcommand{\pathto}[3]{#1\overset{#2}{\nyoroto} #3}
\newcommand{\pathtoD}[3]{#1\overset{#2}{\leadsto\!\leadsto} #3}
\newcommand{\pathtoT}[3]{#1\overset{#2}{\leadsto\!\leadsto\!\leadsto} #3}
\def\varnabla{\stackrel{\rotatebox{180}{$\varDelta$}}{\text{ }}}
\def\adX{\mathrm{ad}X}
\def\lala{\la\!\la}
\def\rara{\ra\!\ra}
\def\jmath{\mathscr{J}}
\newcommand\mathj{\mathscr{J}}
\def\nyoroto{{\rightsquigarrow}}
\def\tilchi{{\tilde{\chi}}}

\def\bkappa{{\boldsymbol \kappa}}
\def\tilbchi{\tilde{\boldsymbol \chi}}
\def\Li{{\mathsf{Li}}}
\def\scL{{\mathscr{L}}}
\def\scLi{{\mathscr{L}i}}
\def\sLL{{\mathsf{L}}}
\def\cY{{\mathcal{Y}}}
\def\cX{{\mathcal{X}}}
\def\cy{{\tilde{y}}}

\def\wt{{\mathsf{wt}}}
\def\sR{{\mathscr{\large R}}}
\def\sM{{\mathsf{\large M}}}
\def\scM{{\mathscr{\large M}}}
\def\DD{{\mathsf{\large D}}}
\def\PP{{\mathsf{\large P}}}
\def\EE{{\mathsf{\large E}}}

\def\CC{{\mathsf{\large C}}}

\def\ll{{\ell i}}
\def\eps{{\varepsilon}}

\def\bK{{\overline{K}}}

\def\ff{{\frak f}}

\def\la{{\langle}}
\def\ra{{\rangle}}

\def\ovec#1{\overrightarrow{#1}}
\def\isom{\,{\overset \sim \to  }\,}

\def\Z{{\Bbb Z}}

\def\C{{\Bbb C}}

\def\Q{{\Bbb Q}}
\def\bQ{{\overline{\Bbb Q}}}

\def\Lie{\mathrm{Lie}}

\def\Gal{\mathrm{Gal}}

\def\eps{{\varepsilon}}

\def\bchi{{\boldsymbol \chi}}
\def\tilbchi{{\tilde{\bchi}}}
\def\brho{{\boldsymbol \rho}}

\def\CHplus{\underset{\mathsf{CH}}{\oplus}}

\begin{document}

\title{On distribution formulas \\
for complex and $\ell$-adic polylogarithms}


\author{Hiroaki Nakamura and Zdzis{\l}aw Wojtkowiak}

\dedication{Dedicated to the memory of Professor Jean-Claude Douai} 


\address{Hiroaki Nakamura: 
Department of Mathematics, 
Graduate School of Science, 
Osaka University, 
Toyonaka, Osaka 560-0043, Japan}
\email{nakamura@math.sci.osaka-u.ac.jp}

\address{Zdzis{\l}aw Wojtkowiak: 
Universit\'e de Nice-Sophia Antipolis, D\'ept. of Math., 
Laboratoire Jean Alexandre Dieudonn\'e, 
U.R.A. au C.N.R.S., No 168, Parc Valrose -B.P.N${}^\circ$ 71,
06108 Nice Cedex 2, France}
\email{wojtkow@math.unice.fr}

\maketitle

\markboth{H.Nakamura, Z.Wojtkowiak}
{On distribution formula
for complex and $\ell$-adic polylogarithms}


\begin{abstract}
We study an $\ell$-adic Galois analogue of the distribution formulas for 
polylogarithms with special emphasis on path dependency and 
arithmetic behaviors. 
As a goal, we obtain a notion of certain universal
Kummer-Heisenberg measures that enable
interpolating the $\ell$-adic polylogarithmic 
distribution relations for all degrees.
\end{abstract}


\tableofcontents

\renewcommand{\thefootnote}{}
\footnote{
This article has appeared in the proceedings volume
``Periods in Quantum Field Theory and Arithmetic" 
(J.~Burgos Gil, K.~Ebrahimi-Fard, H.~Gangl eds),
[Conference proceedings ICMAT-MZV 2014]
Springer Proceedings in Mathematics \& Statistics
{\bf 314} (2020), pp.593--619.

}

\vspace{-1cm}

\section{Introduction}

One of the most important and useful functional equations of 
classical complex polylogarithms is a series of 
distribution relations
\begin{equation}
\label{classical}
Li_k(z^n)=n^{k-1}
\left(\sum_{i=0}^{n-1}
Li_k(\zeta_n^iz)
\right)
\qquad
(\zeta_n=e^{2\pi i/n}).
\end{equation}
J.~Milnor \cite[(7),\,(32)]{M83} says 
that a function $\scL_s(z)$ has (multiplicative) Kubert identities 
of degree $s\in\C$,
if it satisfies
\begin{equation}
\label{milnor}
\scL_s(z)=n^{s-1}\sum_{w^n=z}
\scL_s(w) 
\end{equation}
for every positive integer $n$.
The aforementioned classical identity (\ref{classical}) 
for $Li_k(z)$ is, of course, 
a typical example of Kubert identity of degree $k$, assuming,
however, correct choice of branches of the multivalued function 
$Li_k$ in all terms of the identity.
To avoid the ambiguity of branch choice, we would rather 
consider $\scL_s(z)$ as a function  $\scL_s(z;\gamma)$
of paths $\gamma$  on $\bold P^1-\{0,1,\infty\}$
from the unit vector $\ovec{01}$ to $z$.
The main aim of this paper is to study generalizations of the
above distribution relation for 
multiple polylogarithms and their $\ell$-adic Galois analogs
($\ell$-adic iterated integrals) with special emphasis on
path dependency.

Let $K$ be a subfield of $\C$ with the algebraic closure $\bK\subset\C$.
The {\it $\ell$-adic polylogarithmic characters}
$$
\tilchi_k^z: G_K\to \Z_\ell \qquad (k=1,2,\dots)
$$ 
are introduced in \cite{NW1} as $\Z_\ell$-valued 1-cochains 
on the absolute 
Galois group $G_K:=\Gal(\bK/K)$
for any given path 
$\gamma$ from 
$\ovec{01}$ to 
a $K$-rational point $z$ on $\bold P^1-\{0,1,\infty\}$.
Our study in [NW2] showed that $\tilchi_k^z:G_K\to \Z_\ell$
behave nicely as 
$\ell$-adic analogues of the classical polylogarithms $Li_k(z)$.
The $\ell$-adic polylogarithms and $\ell$-adic iterated integrals
are $\Q_\ell$-valued 
variants (and generalizations) of the above 1-cochains
$\tilchi_k^z: G_K\to \Z_\ell$. (See \S2 and \S 3 for their
precise definitions.)
We will give a geometrical proof of distribution relations for 
classical multiple polylogarithms and their $\ell$-adic
analogues in considerable generality. 
In particular, we will obtain several versions of 
Kubert identities with explicit path systems
for:
\begin{itemize}
\item
classical multiple polylogarithms (Theorem \ref{complexGeneralFormula}), 
\item
$\ell$-adic iterated integrals (Proposition \ref{l-adicGeneralYwords}, 
Theorem \ref{l-adicGeneralFormula}), 
\item
$\ell$-adic polylogarithms and polylogarithmic characters 
(Theorem \ref{l-adicpolythm}, Corollary \ref{tilchi-homogeneous}).
\end{itemize}

The polylogarithm is interpreted as a certain coefficient of 
an extension of the Tate module by the logarithm sheaf
arising from the fundamental group of $V_1:=\bold P^1-\{0,1,\infty\}$.
The motivic construction dates back to the fundamental work
of Beilinson-Deligne \cite{BD}, Huber-Wildeshaus \cite{HW98}
(see also \cite{HK} \S 6 and references therein
for more recent generalizations).
In this article, we mainly work on the $\ell$-adic realization
which forms a $\Z_\ell$- or $\Q_\ell$-valued 1-cochain on the
Galois group $G_K$.
In the collaboration \cite{DW} of the last author with J.-C.Douai, 
it was shown that certain linear combinations 
of $\ell$-adic polylogarithms at various points 
give rise to 1-cocycles on $G_K$, which lead to 
an $\ell$-adic version of Zagier's conjecture.
See also Remark \ref{relationBD} and \cite{NSW} \S 3.2

We will intensively make use of a system of 
simple cyclic covers 
$V_n:=\bold P^1-\{0,\mu_n,\infty\}$ 
over $V_1= \bold P^1-\{0,1,\infty\}$,
where 
$\mu_n$ is the group of $n$-th roots of unity
$\{1,\zeta_n,\dots,\zeta_n^{n-1}\}$
($\zeta_n:=e^{2\pi i/n}$),
and
$\{0,\mu_n,\infty\}$ denotes $\{0,\infty\}\cup\mu_n$
by abuse of notation.
We consider the family of cyclic coverings $V_n\to V_1$ and open immersions
$V_n\hookrightarrow V_1$ together with induced relations between their
fundamental group(oid)s.
Our basic idea is to understand the distribution relations of polylogarithms
as the ``trace property'' of relevant coefficients (``iterated integrals'')
arising in those fundamental groups.

As observed in \cite{NW2} and will be seen in  \S 3 below, 
unlike in the classical complex case,
there generally occur lower degree terms in $\ell$-adic case 
when a distribution relation is naively derived. 
This problem prevents artless approaches to 
$\ell$-adic Kubert identities 
i.e., distribution formulas of homogeneous form (with no lower degree terms).
Our line of studies in \S 2-6 will lead us to understand why and how 
to make use of $\Q_\ell$-paths ($\ell$-adic paths with `denominators')
to eliminate such lower degree terms dramatically.
Consequently in \S 7, as a primary goal of this paper, 
we arrive at introducing 
a generalization of the Kummer-Heisenberg measure
of \cite{NW1} so as to interpolate 
those $\ell$-adic distribution relations of polylogarithms
for all degrees.

\begin{Remark}
We have already studied in \cite{W2} and \cite{W4}
the distribution relations for those 
$\ell$-adic polylogarithms under certain
restricted assumptions
(see \cite[Prop.\,11.1.4]{W2} for $\ell$-adic dilogarithms,
\cite[Cor. 11.2.2, 11.2.4]{W2} for $\ell$-adic polylogarithms on
restricted Galois groups,
and \cite[Th.\,2.1]{W4} for $\ell$-adic polylogarithmic characters
with $\ell\nmid n$).
\end{Remark}

\noindent
{\bf Basic setup, notations and convention}:
\medskip

Below, we understand that all algebraic varieties are 
geometrically connected over a fixed field
$K\subset \C$ and that all morphisms
between them are $K$-morphisms.
A path on a $K$-variety $V$ is a topological path on $V(\C)$
or an \'etale path on $V\otimes\bK$ whose distinction will be obvious
in contexts.
The notation $\gamma: x\nyoroto y$ means 
a path from $x$ to $y$,
and write $\gamma_1\gamma_2$ for the composed path tracing
$\gamma_1$ first and then $\gamma_2$ afterwards. 
We write $\chi:G_K\to\Z_\ell^\times$ for
the $\ell$-adic cyclotomic character ($\ell$: a fixed prime).
The Bernoulli polynomials $B_k(T)$ $(k=0,1,\dots)$ 
are defined by the generating function
$\frac{ze^{Tz}}{e^z-1}=\sum_{k=0}^\infty B_k(T)\frac{z^k}{k!}$,
and the Bernoulli numbers are set as $B_k:=B_k(0)$.
For a vector space $H$, we write $H^\ast$ for 
its dual vector space.

Assume $K\supset \mu_n$. 
We shall be concerned with two kinds of standard morphisms 
defined by
\begin{equation*}
\begin{cases}
\jmath_\zeta:V_n\hookrightarrow V_1
&\jmath_\zeta(z)=\zeta z\qquad
(\zeta\in\mu_n)
; \\
\pi_n:V_n\to V_1
&\pi_n(z)=z^n.
\end{cases}
\end{equation*}

\vspace{-5mm}
\begin{figure}[H]
\begin{center}
\includegraphics[width=7cm, bb=0 0 640 400]{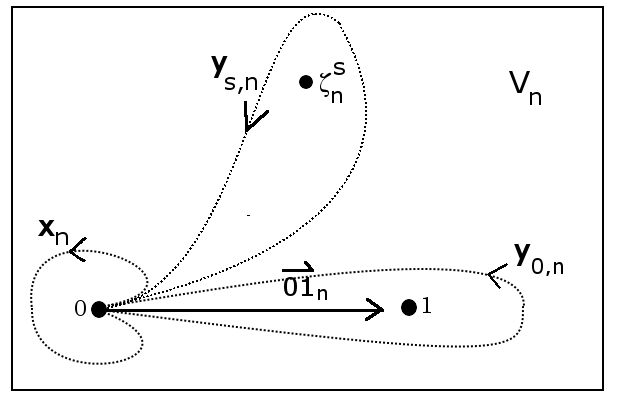}
$\begin{matrix}
\overset{\pi_n}{\rightarrow} \\
\overset{\mathj_\zeta}{\hookrightarrow}\\
{}\\{}\\
{}\\
{}\\
\end{matrix}
\quad
$
\includegraphics[width=7.3cm, bb=0 0 640 400]{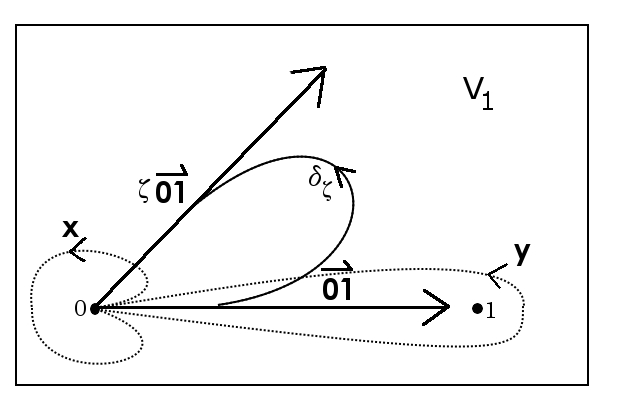}
\end{center}
\end{figure}

\vspace{-15mm}

\noindent
As easily seen, each $\jmath_\zeta$ is an open immersion,
while $\pi_n$ is an $n$-cyclic covering.
Write $\ovec{01}_n$ for the
tangential base point represented by the
unit tangent vector on $V_n$.
Since $\mathj_1:V_n\hookrightarrow V_1$ maps $\ovec{01}_n$
to $\ovec{01}_1$(often written just $\ovec{01}$), it induces
the surjection homomorphism
\begin{equation}
\pi_1(V_n,\ovec{01}_n)\twoheadrightarrow \pi_1(V_1,\ovec{01}).
\end{equation}
On the other hand, although the image $\pi_n(\ovec{01}_{n})$ 
is not exactly the same as $\ovec{01}_{1}$ as a tangent vector, 
they give the same tangential base point on $V_1$
in the sense that they give
equivalent fiber functors on the Galois category 
of finite \'etale covers of $V_1$.
Henceforth, for simplicity, we shall 
regard $\pi_n(\ovec{01}_{n})=\ovec{01}_{1}=\ovec{01}$,
and regard $\pi_1(V_n,\ovec{01}_n)$ as a subgroup of 
$\pi_1(V_1,\ovec{01})$ by the homomorphism
\begin{equation}
\label{inclusion}
\pi_1(V_n,\ovec{01}_n)\hookrightarrow \pi_1(V_1,\ovec{01})
\end{equation}
induced from $\pi_n$.

For each $\zeta\in\mu_n$,
introduce a path
$\delta_\zeta:\ovec{01}\nyoroto\zeta\ovec{01}
=\jmath_\zeta(\ovec{01}_{n})$ on $V_1$
to be the arc from $\ovec{01}$ to $\zeta\ovec{01}$
anti-clockwise oriented.
Using the path $\delta_\zeta$, we obtain the identification
$\pi_1(V_1,\ovec{01})\isom \pi_1(V_1,\zeta\ovec{01})$.

Let $x$, $y$ be standard loops based at $\ovec{01}_{1}$ 
on $V_1=\bold P^1-\{0,1,\infty\}$
turning around the punctures $0$, $1$ once anticlockwise
respectively.
We introduce loops $x_n$, $y_{0,n}\dots,y_{n-1,n}$ 
based at $\ovec{01}_{n}$ on $V_n$ characterized by:
\begin{equation*}
\begin{cases}
x_n&:=\pi_n^{-1} (x^n)=\mathj_1^{-1}(x), \\
y_{s,n}&:=\mathj_1^{-1}(\delta_\zeta)\cdot
\jmath_{\zeta^{-1}}^{-1}(y)\cdot
\mathj_1^{-1}(\delta_\zeta)^{-1}
\qquad (\zeta=e^{\frac{2\pi i s}{n}},\ s=0,\dots,n-1)
\end{cases}
\end{equation*}
so that $x_n$, $y_{0,n},\dots, y_{n-1,n}$
freely generate $\pi_1(V(\C),\ovec{01}_n)$.

Note that, in view of the above inclusion (\ref{inclusion}),
we have the identifications:
\begin{equation}
x_n=x^n, \quad
y_{s,n}=x^syx^{-s}.
\end{equation}


\section{Complex distribution relations}

For $n=1,2,...$, 
let
$$
\omega(V_n):=\frac{dz}{z}\otimes\biggl(\frac{dz}{z}\biggr){\!\rule{0pt}{2.5ex}}^\ast
+
\sum_{i=0}^{n-1}
\frac{dz}{z-\zeta_n^i}\otimes
\biggl(\frac{dz}{z-\zeta_n^i}\biggr){\!\rule{0pt}{2.5ex}}^\ast
\in
\Omega^1_{\log}(V_n)\otimes \Omega^1_{\log}(V_n)^\ast
$$
be the canonical one-form on $V_n$.
Traditionally, we set
$$
X_n:=\biggl(\frac{dz}{z}\biggr){\!\rule{0pt}{2.5ex}}^\ast
\quad\text{and}\quad
Y_{i,n}:=\biggl(\frac{dz}{z-\zeta_n^i}\biggr){\!\rule{0pt}{2.5ex}}^\ast.
$$
Let $\sR_n:=\C\lala X_n, Y_{i,n}\mid 0\le i<n\rara$
be the non-commutative algebra of formal power series
over $\C$ generated by non-commuting variables
$X_n$ and $Y_{i,n}$ $(0\le i<n)$.
Consider the trivial bundle 
$$
\sR_n\times V_n\to V_n
$$
equipped with the (flat) connection 
$\varnabla: \Phi\mapsto d\Phi-\Phi\, \omega(V_n)$
for smooth functions $\Phi:V_n\to \sR_n$.
For a piecewise smooth path $\gamma:[0,1]\to V_n$ from
$\gamma(0)=a$ to $\gamma(1)=z$, 
let  $\Phi:[0,1]\to \sR_n$ be the solution to
the differential equation $d\Phi=\Phi\,\omega(V_n)$
pulled back on $\gamma$ with $\Phi(0)=1$ and define
$\Lambda(\pathto{a}{\gamma}{z})\in\sR_n$ to be $\Phi(1)$.
(Cf.\,\cite{H} \S 2, \cite{W96} \S 1; 
we here follow Hain's path convention in loc.\,cit.)
In the case $a$ being the tangential base point
$\overrightarrow{01}$, we interpret
$\Lambda(\pathto{\ovec{01}}{\gamma}{z})$
in a suitable manner introduced 
in \cite{De}, \cite[\S 3.2]{W97}.

Let $\sM_n$ be the set of all monomials (words)
in $X_n$ and $Y_{i,n}$ $(0\le i<n)$.
Then, we can expand
\begin{equation}
\label{complexMagnusExpansion}
\Lambda(\pathto{a}{\gamma}{z})=1+\sum_{w\in\sM_n}
\Li_w(\pathto{a}{\gamma}{z})\cdot w
\end{equation}
in $\sR_n$.
If $w=X_n^{a_0}Y_{i_1,n}X_n^{a_1}\cdots Y_{i_k,n}X_n^{a_k}$,
then
\begin{equation}
\label{complexIteratedIntegral}
\Li_w(\pathto{a}{\gamma}{z})
=
\int_{a,\gamma}^z
\underbrace{\frac{dz}{z}\cdots\frac{dz}{z}}_{a_0}\cdot\frac{dz}{z-\zeta_n^{i_1}}
\cdots
\frac{dz}{z-\zeta_n^{i_k}}\cdot
\underbrace{\frac{dz}{z}\cdots\frac{dz}{z}}_{a_k},
\end{equation}
the iterated integral along $\gamma$.

\begin{Definition}
\label{def2.1}
For a word $w=X_n^{a_0}Y_{i_1,n}X_n^{a_1}\cdots Y_{i_k,n}X_n^{a_k}$,
we define its $X$-weight by
$$
\wt_X(w)=a_0+\cdots+a_k.
$$
\end{Definition}

Let the cyclic cover 
\begin{equation}
\pi_{rn,r}:V_{rn}\longrightarrow V_r
\end{equation}
be given by $\pi_{rn,r}(z)=z^n$.
Then, we have
$$
\Bigl[\mathrm{Id}\otimes (\pi_{rn,r})_\ast\Bigr](\omega(V_{rn}))
=\Bigl[(\pi_{rn,r})^\ast\otimes\mathrm{Id}\Bigr](\omega(V_r)).
$$
This implies that the induced map from
$(\pi_{rn,r})_\ast$ on complete tensor algebras
(denoted by the same symbol):
\begin{eqnarray*}
\xymatrix{
\hat T(\Omega^1_{\log}(V_{rn})^\ast)
\ar[r]\ar@{=}[d] & 
\hat T(\Omega^1_{\log}(V_r)^\ast)
\ar@{=}[d]
    \\
\C\lala X_{rn}, Y_{j,rn}\mid 0\le j<rn\rara 
\ar[r] &
\C\lala X_r, Y_{i,r}\mid 0\le i<r\rara
}
\end{eqnarray*}
preserves the associated power series:
\begin{equation}
\label{powerseriesformula}
(\pi_{rn,r})_\ast\left(
\Lambda(\pathto{\ovec{01}}{\gamma}{z})
\right)
=
\Lambda(\pathtoT{\ovec{01}}{\pi_{rn,r}(\gamma)}{z^n}).
\end{equation}

Note that 
\begin{equation}
\label{complexpreserve}
\begin{cases}
&(\pi_{rn,r})_\ast (X_{rn})=nX_r, \\
&(\pi_{rn,r})_\ast (Y_{j,rn})=Y_{i,r}  \quad
(i\equiv j \ \mathrm{mod} \ r).
\end{cases}
\end{equation}

\begin{Definition}
\label{def2.2}
For $w\in\sM_{rn}$, we mean by ($w$ mod $r$)
the word in $\sM_{r}$ obtained by replacing
each letter $X_{rn}$, $Y_{j,rn}$ $(0\le j<rn)$
appearing in $w$ by
$X_{r}$, $Y_{i,r}$ (where $i$ is an integer with
$0\le i<r$, $i\equiv j$ mod $r$)   
respectively. 
If $r$ is a common divisor of $m$ and $n$,  
$w\in\sM_m$, $w'\in\sM_{n}$
and ($w$ mod $r) = (w'$ mod $r$), 
then we shall write
$$
w\equiv w' \quad \mathrm{mod} \ r.
$$
\end{Definition}

\begin{Theorem}
\label{complexGeneralFormula}
Notations being as above, let $\gamma$ be a path
on $V_{rn}$ from $\ovec{01}$ to a point $z$.
Then, for any word $w\in\sM_r$, 
we have the distribution relation
$$
\Li_w(\pathtoT{\ovec{01}}{\pi_{rn,r}(\gamma)}{z^n})
=
n^{\wt_X(w)}
\sum_{\substack{u\in\sM_{rn}\\ u\equiv w\ \mathrm{mod}\ r}}
\Li_u(\pathto{\ovec{01}}{\gamma}{z}).
$$
\end{Theorem}

\begin{proof}
The theorem follows immediately from the formula
(\ref{powerseriesformula}):
Write 
$
\Lambda(\pathto{\ovec{01}}{\gamma}{z})=1+\sum_{u\in\sM_{rn}}
\Li_u(\pathto{\ovec{01}}{\gamma}{z})\cdot u
$
in $\sR_{rn}$.
Applying (\ref{powerseriesformula}), we obtain
$$
1+ \sum_{w\in\sM_r}
\Li_w(\pathtoT{\ovec{01}}{\pi_{rn,r}(\gamma)}{z^n})\cdot w
=
1+\sum_{u\in\sM_{rn}} 
\Li_u(\pathto{\ovec{01}}{\gamma}{z})\cdot (\pi_{rn,r})_\ast(u).
$$  
Given any specific $w\in \sM_r$ in LHS, collect from RHS
all the coefficients of $(\pi_{rn,r})_\ast(u)$ 
for those $u$ satisfying $(u$ mod $r)=w$. 
Noting that $(\pi_{rn,r})_\ast(u)=n^{\wt_X(u)}w=n^{\wt_X(w)} w$
for them, we settle the assertion of the theorem.
\end{proof}

The above theorem generalizes the distribution
relation (\ref{classical}) for the classical polylogarithm
$Li_k(z)_\gamma$ along the path $\gamma:\ovec{01}\nyoroto z$.
Indeed, in the notation above, 
since $\frac{dz}{1-z}=-\frac{dz}{z-1}$, we may identify
$$
Li_k(z)_\gamma=
-\Li_{YX^{k-1}}(\pathto{\ovec{01}}{\gamma}{z}).
$$
Applying the theorem to the
special case $\pi_{n,1}:V_n\to V_1$, $w=YX^{k-1}$
where $Y=Y_{0,1}$, $X=X_{1}$,
we obtain
$$
\int_{\ovec{01}, \pi_{n,1}(\gamma)}^{z^n}
\frac{dz}{z-1}\cdot
\underbrace{\frac{dz}{z}\cdots\frac{dz}{z}}_{k-1}
=
n^{k-1}
\sum_{\zeta\in\mu_n}
\int_{\ovec{01},\gamma}^z
\frac{dz}{z-\zeta}\cdot
\underbrace{\frac{dz}{z}\cdots\frac{dz}{z}}_{k-1}.
$$
Each term of RHS turns out to be $Li_k(\zeta z)$
along the path 
$\delta_\zeta \cdot\mathj_\zeta(\gamma):\ovec{01}\nyoroto 
\zeta\ovec{01}\nyoroto \zeta z$,
after integrated by substitution $z\to \zeta z$.
Noting that the integration here over 
$\delta_\zeta: \ovec{01}\nyoroto \zeta\ovec{01}$
vanishes (cf. \cite{W97} \S 3),
we obtain (\ref{classical}) with path system specified
as follows:
\begin{equation}
\label{refinedclassical}
Li_k(z^n)_{\pi_{n,1}(\gamma)}
=
n^{k-1}
\sum_{\zeta\in\mu_n}
Li_k(\zeta z)_{\delta_\zeta\cdot\mathj_\zeta(\gamma)}.
\end{equation}


\section{$\ell$-adic case (general)}

We shall look at the $\ell$-adic analogue of the previous section
by recalling the following construction which essentially
dates back to \cite{W04a}.
Let $K\subset \C$ and consider 
$$
\pi_1^{\ell}(V_n\otimes\overline{K},\ovec{01}),
$$
the pro-$\ell$ (completion of
the \'etale) fundamental group of $V_n\otimes \overline{K}$.
It is easy to see 
the loops 
$x_n$, $y_{1,n},\dots,y_{n-1,n}$ 
introduced in \S 1 form 
a free generator system of the pro-$\ell$ fundamental group.
Consider the (multiplicative) Magnus embedding 
into the ring of non-commutative power series
\begin{equation}
\label{MagnusEmbedding}
\iota_{\Q_\ell}:
\pi_1^{\ell}(V_n\otimes\overline{K},\ovec{01})
\longhookrightarrow
\Q_\ell\lala X_n, Y_{i,n}\mid 0\le i \le n-1 \rara
\end{equation}
defined by $\iota_{\Q_\ell}(x_n)=\exp(X_n)$, 
$\iota_{\Q_\ell}(y_{i,n})=\exp(Y_{i,n})$ (cf. \cite[15.1]{W3}).
For simplicity, we often identify elements of 
$\pi_1^{\ell}(V_n\otimes\overline{K},\ovec{01})$
with their images by $\iota_{\Q_\ell}$.
Let us write again $\sM_n$ for the set of monomials in $X_n,Y_{i,n}$
$(i=0,\dots,n-1)$ (although variables have different
senses from the previous section where they 
were duals of differential forms). 
We shall also employ the usage `$\wt_X(w)$' and 
`$w\equiv w'$ mod $r$' by following the same manners
as Definitions \ref{def2.1} and \ref{def2.2}.

Recall that we have a canonical Galois action
$G_K$ on (\'etale) paths on $V_n\otimes\overline{K}$
with both ends at $K$-rational (tangential) points.
Given a path $\gamma$ from such a point $a$
to a point
$z\in V_n(K)$, we set, for any $\sigma\in G_K$,
\begin{equation} \label{def-of-ff}
\ff_\sigma^\gamma:=\gamma\cdot\sigma(\gamma)^{-1}
\in 
\pi_1^{\ell}(V_n\otimes\overline{K},a ),
\end{equation}
where the RHS is understood to be 
the image in the pro-$\ell$ quotient.
When $a=\ovec{01}$, we expand 
$\ff_\sigma^\gamma\in \pi_1^{\ell}(V_n\otimes\overline{K},\ovec{01})$
in the form
\begin{equation} \label{LiDef}
\ff_\sigma^\gamma
=1+\sum_{w\in\sM_n}
\Li_w(\pathto{\ovec{01}}{\gamma}{z})(\sigma) \cdot w
\end{equation}
in 
$\Q_\ell\lala X_n, Y_{i,n}\mid 0\le i \le n-1 \rara$,
and associating the coefficient 
$$
\Li_w(\pathto{\ovec{01}}{\gamma}{z})(\sigma)
:=
\mathsf{Coeff}_w(\ff_\sigma^\gamma)
$$
of $w\in \sM_n$ to $\sigma\in G_K$, 
we define the $\ell$-adic Galois 1-cochain
$$
\Li_w(\pathto{\ovec{01}}{\gamma}{z})
\left(=\Li_w^{(\ell)}(\pathto{\ovec{01}}{\gamma}{z})
\right)
:G_K\to\Q_\ell
$$
for every monomial $w\in \sM_n$.
We call each $\Li_w(\pathto{\ovec{01}}{\gamma}{z})$
the $\ell$-adic iterated integral
associated to $w\in \sM_n$ and to the
path $\gamma$ on $V_n$.

\begin{Remark}
The above naming `$\ell$-adic iterated integral' is intended to
be an analog of the iterated integral appearing in the
complex case (\ref{complexMagnusExpansion}),
(\ref{complexIteratedIntegral}).
They represent general coefficients of the associator in
the Magnus expansion. Conceptually, the associator lies 
in the pro-unipotent hull of the fundamental group and the
monodromy information encoded in
the total set of them is equivalent to that encoded in
the general coefficients with respect to any 
fixed Hall basis of the corresponding Lie algebra. 
This line of formulation was, in fact, taken up, e.g., in \cite{W04a} \S 5.
However for the purpose of pursuing the distribution formulas  in the present paper, 
the simple form of trace properties (\ref{powerseriesformula}), (\ref{complexpreserve}) 
along the cyclic
coverings $\pi_{rn,r}:V_{rn}\longrightarrow V_r$ is most essential.
This is why we start with Magnus expansions $\ff_\sigma^\gamma$ 
in $\Q_\ell\lala X_n,Y_{i,n} \rara_i$ rather than with Lie expansions of $\log\ff_\sigma^\gamma$ 
with respect to a Hall basis in $\Lie\lala X_n, Y_{i,n}\rara_i$. 
But we shall discuss their relations in the polylogarithmic part of $n=1$ in \S 4.
\end{Remark}

Now, as in \S 2, let us consider the morphism
$\pi_{rn,r}:V_{rn}\to V_{r}$ given by
$\pi_{rn,r}(z)=z^n$ for $n,r>0$, and let 
$\gamma$ be a path on $V_{rn}$ from
$\ovec{01}$ to a $K$-rational point $z$.
By our construction, the $\ell$-adic analogue of the
equality (\ref{powerseriesformula}) holds, i.e., 
$\pi_{rn,r}$ preserves the $\ell$-adic associators:
\begin{equation}
\label{l-adicGeneralFormula}
(\pi_{rn,r})_\ast(\ff_\sigma^{\,\gamma})=
\ff_\sigma^{\, \pi_{rn,r}(\gamma)}
\quad (\sigma\in G_K).
\end{equation}
However, unlike the complex case
(\ref{complexpreserve}), $\pi_{rn,r}$ does
not preserve the expansion coefficients
homogeneously, i.e., it maps as
\begin{equation}
\label{eq3.2}
\begin{cases}
(\pi_{rn,r})_\ast(X_{rn})&=nX_r, \\
(\pi_{rn,r})_\ast(Y_{j,rn})
&=\exp(kX_r)Y_{i, r}\exp(-kX_r)
\qquad(j=i+kr,\ 0\le i<r).
\end{cases}
\end{equation}
\begin{proof}[Proof of $(\ref{eq3.2})$]
Note that the cyclic projections 
$\pi_{rn,r}$ identify $\{\pi_1(V_n)\}_n$ 
as a sequence of subgroups
of $\pi_1(V_1)$ as in (\ref{inclusion}),
and regard
$x_{rn}=x_r{}^n=x^{rn}$, 
$y_{j,rn}=x^j y x^{-j}=(x^r)^kx^iyx^{-i}(x^r)^{-k}=x_r^ky_{i,r}x_r^{-k}$.
Although $\pi_{rn,r}$ does not keep injectivity
on the complete envelops, it does induce a functorial homomorphism
on them. The formula follows then from $x_n=\exp(X_n)$,
$y_{s,n}=\exp(Y_{s,n}).$
\end{proof}
This causes generally (lower degree) error terms to
appear in distribution relations for $\ell$-adic
iterated integrals.  

Still, if we restrict ourselves to the words whose 
$X$-weights are zero, we have the following

\begin{Proposition}
\label{l-adicGeneralYwords}
Notations being as above, if $w\in \sM_{r}$
is a word with $\wt_X(w)=0$, i.e.,  
of the form
$w=Y_{i_k,r}\cdots Y_{i_1,r}$, then
it holds that 
$$
\Li_w(\pathtoT{\ovec{01}}{\pi_{rn,r}(\gamma)}{z^n})(\sigma)
=
\sum_{\substack{u\in\sM_{rn}\\ u\equiv w
\ \mathrm{mod}\ r}}
\Li_u(\pathto{\ovec{01}}{\gamma}{z})(\sigma)
\qquad
(\sigma\in G_K).
$$
\end{Proposition}

\begin{proof}
In the expansion of 
$(\pi_{rn,r})_\ast(\ff_\sigma^\gamma)=
\ff_\sigma^{\,\pi_{rn,r}(\gamma)}$,
the contributions to the coefficient of $w$
come only from the first `$Y$-only' term
of each $u\in\sM_{rn}$ with $u\equiv w$ 
mod $r$. 
The proposition follows from this observation.
\end{proof}

\begin{Remark}
In the $\ell$-adic Galois case, the distribution
relations of Proposition
\ref{l-adicGeneralYwords}
are used in \cite{W5} to construct measures
on $\Z_\ell^r$ which generalize the measure
on $\Z_\ell$ in \cite{NW1}.
The general distribution formula analogous to
Theorem \ref{complexGeneralFormula} for arbitrary 
words in $\sM_r$ hold
only up to lower degree terms in the
$\ell$-adic Galois case.
More generally, any covering maps between smooth
algebraic varieties will give 
some kind of distribution relations. 
\end{Remark}

\section{$\ell$-adic polylogarithms (review)}

Henceforth, we shall closely look at the case of 
$\ell$-adic polylogarithm where $r=1$ and only
those words $w\in\sM_1$ involving $Y_{0,1}$ only once
are concerned, in the setting of the previous section.
For simplicity, we write $x:=x_1$, $y:=y_{0,1}$
and $X:=\log(x)$, $Y:=\log(y)$,
and will be concerned with those coefficients
of the words $YX^{k-1}$ of $\ff_\sigma^\gamma$.

Let us recall some basic facts from \cite{NW1}, \cite{NW2}.
We introduced, for any path 
$\gamma:\ovec{01}\nyoroto z$
on $V_1=\bold P^1-\{0,1,\infty\}$,
the {\it $\ell$-adic polylogarithms}
\begin{equation}
\label{ladicpolylog}
\ll_m(z,\gamma):G_K\to \Q_\ell
\end{equation}
(with regard to the fixed free generator system 
$\{x,y\}$ of $\pi_1^\ell(V_1\otimes\bK,\ovec{01})$)
to be the Lie expansion coefficients of the associator
$\ff_\sigma^\gamma=\gamma\cdot\sigma(\gamma)^{-1}$ for
$\sigma\in G_K$ modulo the ideal $I_Y$ of Lie monomials
including $Y$ twice or more:
\begin{equation}
\label{ladicpolydef}
\log(\ff_\sigma^\gamma)^{-1}
\equiv\rho_z(\sigma) X +
\sum_{m=1}^\infty
\ll_m(z,\gamma)(\sigma)
\mathrm{ad}(X)^{m-1}(Y)
\quad\text{ mod }I_Y.
\end{equation}
Here, $\rho_z:G_K\to\Z_\ell(1)$ designates the Kummer
1-cocycle for power roots of $z$ along $\gamma$. 
Note, however, that the other coefficients 
$\ll_m(z,\gamma)(\sigma)\in\Q_\ell$ 
are generally not valued in $\Z_\ell$ due to applications of 
$\log$ respectively to $x$, $y$ and 
$\ff_\sigma^\gamma\in \pi_1^\ell(V_1\otimes\bK,\ovec{01})$.
In fact, we can bound the denominators of $\ll_m(z,\gamma)(\sigma)$
by relating them with more explicitly
defined $\Z_\ell$-valued 1-cochains called
the {\it $\ell$-adic polylogarithmic characters}
\begin{equation}
\label{ladicpolylogcharacter}
\tilchi_m^z(=\tilchi_m^{z,\gamma}):G_K\to \Z_l \qquad (m\ge 1)
\end{equation}
defined by the Kummer properties for $n\ge 1$:
\begin{equation}
\label{defpolychar}
\zeta_{\ell^n}^{\tilchi_m^z(\sigma)}=
\sigma\left(
\prod_{a=0}^{\ell^n-1}
(1-\zeta_{\ell^n}^{\chi(\sigma)^{-1}a}z^{1/\ell^n})^{\frac{a^{m-1}}{\ell^n}}\right)
\Bigm/
\prod_{a=0}^{\ell^n-1}
(1-\zeta_{\ell^n}^{a+\rho_z(\sigma)}z^{1/\ell^n})^{\frac{a^{m-1}}{\ell^n}},
\end{equation}
where $(1-\zeta_{\ell^n}^\alpha z^{1/\ell^n})^{\frac{\beta}{\ell^n}}$ 
means the $\beta$-th power of
a carefully chosen $\ell^n$-th root of $(1-\zeta_{\ell^n}^\alpha z^{1/\ell^n})$ 
along $\gamma$.
It is shown in \cite[p.293 Corollary]{NW1} that, for each $\sigma\in G_K$, the
$\ell$-adic polylogarithm
$\ll_m(z,\gamma)(\sigma)\in\Q_\ell$ 
can be expressed by the Kummer- and $\ell$-adic polylogarithmic characters 
$\rho_z(\sigma), \tilchi_m^z(\sigma) \in\Z_\ell$ as follows:
\begin{equation}
\label{explicitNW1}
\ll_m(z,\gamma)(\sigma)=
(-1)^{m+1}
\sum_{k=0}^{m-1}\frac{B_k}{k!}(-\rho_z(\sigma))^k
\frac{\tilchi_{m-k}^z(\sigma)}{(m-k-1)!}
\qquad (m\ge 1).
\end{equation}

One has then the following relations among 
$\ll_m(z,\gamma)(\sigma)\in\Q_\ell$ (\ref{ladicpolylog}),
$\tilchi_m^z(\sigma)\in\Z_\ell$ (\ref{ladicpolylogcharacter})
and 
$\Li_{YX^{m-1}}(\pathto{\ovec{01}}{\gamma}{z})(\sigma)\in\Q_\ell$
(\S 3):
\begin{Proposition}
\label{prop4.1}
(i) Notations being as above, we have
\begin{equation*}
\tilchi_m^z(\sigma)=(-1)^{m+1}(m-1)!
\sum_{k=1}^m
\frac{{\rho_z(\sigma)}^{m-k}}{(m+1-k)!}
\ll_k(z,\gamma)(\sigma)
\quad (m\ge 1).
\end{equation*}
(ii) 
Moreover,  the expansion of
$\ff_\sigma^\gamma$ in $\Q_\ell\lala X,Y\rara$
partly looks like
$$
\ff_\sigma^\gamma=
1+
\sum_{i=1}^\infty \frac{(-\rho_z(\sigma))^i}{i!}X^i
-\sum_{i=0}^\infty \frac{\tilchi_{i+1}^z(\sigma)}{i!} YX^i
+...\text{(other terms)}.
$$
In particular, we have
\begin{equation*}
\Li_{YX^{m-1}}(\pathto{\ovec{01}}{\gamma}{z})(\sigma)
=-\frac{\tilchi_m^z(\sigma)}{(m-1)!}
\quad (m\ge 1).
\end{equation*}
\end{Proposition}

\begin{proof}
(i) follows immediately from inductively reversing
the formula (\ref{explicitNW1}).
(ii) also follows easily from discussions in
\cite[p.284-285]{NW2}:
Suppose 
$\ff_\sigma^\gamma$ has monomial expansion
as 
$$
\ff_\sigma^\gamma=
1+
\sum_{i=1}^\infty c_i \frac{X^i}{i!}
-\sum_{i=0}^\infty d_{i+1} YX^i
+...\text{(other terms)}.
$$
First, from (\ref{ladicpolydef}), we see that
$\ff_\sigma^\gamma\equiv e^{cX}$
modulo $Y=0$ with a constant $c:=-\rho_z(\sigma)$, 
hence that $c_i=c^i$.
Next, to look at the coefficients of monomials of the forms
$X^i$, $YX^i$ ($i=0,1,2,\dots$) closely, we take 
reduction modulo the ideal 
$J_Y:=\la XY, Y^2\ra$ of $\Q_\ell\lala X,Y\rara$.
Observe then the congruence:
\begin{align*}
\log(\ff_\sigma^\gamma)&\equiv
(\ff_\sigma^\gamma -1)\left\{
1-\frac{1}{2}(\ff_\sigma^\gamma-1)+\frac{1}{3}(\ff_\sigma^\gamma-1)^2
-+\cdots\right\} \\
&\equiv
\left( -\sum_{i=0}^\infty d_{i+1} YX^i \right)
\left\{
1-\frac{1}{2}(e^{cX}-1)+\frac{1}{3}(e^{cX}-1)^2
-+\cdots\right\} \\
&\equiv
\left( -\sum_{i=0}^\infty d_{i+1} YX^i \right)
\left\{
\sum_{k=0}^\infty \frac{B_k}{k!} c^kX^k
\right\}\qquad \pmod{J_Y}
\end{align*}
and find that the coefficient of $YX^{m-1}$ 
in $\log (\ff_\sigma^\gamma)$ is 
\begin{equation*}
-\sum_{k=0}^{m-1} \frac{B_k}{k!} c^k d_{m-k}
\tag{$\ast$}
\end{equation*}
for
$m\ge 1$.
\footnote{
Note that there are
misprints in \cite[p.284]{NW2} 
where exponents $\circledast=2,3$
of $(e^{(\log z)X}-1)^\circledast$
should read $\circledast=1,2$ respectively
in the 2nd and 3rd terms in line $-11$.}
On the other hand, the formula (\ref{ladicpolydef}) combined with
(\ref{explicitNW1}) calculates 
the same coefficient, which is $(-1)^{m-1}$-multiple
of that of $\mathrm{ad}(X)^{m-1}(Y)$, as to be
\begin{equation*}
(-1)^{m-1} \ll_m(z,\gamma)
=
\sum_{k=0}^{m-1}\frac{B_k}{k!}(-\rho_z(\sigma))^k
\frac{\tilchi_{m-k}^z(\sigma)}{(m-k-1)!}
\tag{$\ast\ast$}
\end{equation*}
for $m\ge 1$.
Comparing those $(\ast)$ and $(\ast\ast)$ inductively on $m\ge 1$,
we conclude our desired identities $d_{i+1}=-{\tilchi_{i+1}(\sigma)}{/i!}$
($i\ge 0$). 
\end{proof}

\begin{Remark} \label{relationBD}
The $\ell$-adic polylogarithm was constructed 
as a certain lisse $\Q_\ell$-sheaf on $V_1=\bold P^1-\{0,1,\infty\}$ 
as in \cite{BD}, \cite{HW98}, \cite{HK} and \cite{W12}.
The fiber over a point $z\in V_1(K)$ forms a polylogarithmic
quotient torsor of 
$\ell$-adic path classes from $\ovec{01}$ to $z$.
We have the $G_K$-action on the path space whose
specific coefficients are the $\ell$-adic (Galois) polylogarithms
in our sense (\ref{ladicpolylog}), viz., realized 
as $\Q_\ell$-valued 1-cochains on $G_K$.
See also, e.g., \cite{NSW} \S 3 for a concise account from the
viewpoint of non-abelian cohomology 
in a mixed Tate category. 
\end{Remark}

\section{Distribution relations for $\tilchi_m^z$}
Suppose now that $\mu_n\subset K\subset \C$ and 
that we are given a point $z\in V_n(K)$ together with 
a(n \'etale) path
$\gamma: \ovec{01}_n\nyoroto z$ on 
$V_n\otimes \overline{K}=\bold P_{\overline{K}}^1-\{0,\mu_n,\infty\}$. 
We consider the $\ell$-adic polylogarithmic 
characters $\tilchi_m^{z^n}$, $\tilchi_m^{\zeta z}:G_K\to\Z_\ell$
$(\zeta\in \mu_n)$  
along the paths $\pi_n(\gamma):\ovec{01}\nyoroto z^n$
and $\delta_\zeta \mathj_\zeta(\gamma):\ovec{01}\nyoroto \zeta z$
respectively. In this section, we shall show the following
$\ell$-adic analog of the distribution formula:

\begin{Theorem}
\label{MainFormulaInhomo} Notations being as above, we have
$$
\tilchi_k^{z^n}(\sigma)
=
\sum_{d=1}^k
\binom{k-1}{d-1}n^{d-1}
\sum_{s=0}^{n-1}
(s\chi(\sigma))^{k-d}
\tilchi_d^{\zeta_n^s z}(\sigma)
\qquad
(\sigma\in G_K,\ \zeta_n=e^{\frac{2\pi i}{n}},\,0^0=1).
$$
\end{Theorem}


Consider now the $\ell$-adic Lie algebras 
$L_{\Q_\ell}(\ovec{01}_n)$ and $L_{\Q_\ell}(\ovec{01})$
associated to 
$\pi_1^\ell(V_n,\ovec{01}_{n})$ and 
$\pi_1^\ell(V_1,\ovec{01}_{1})$ respectively, 
and set specific elements of them by
$X_n:=\log x_n$, $Y_{i,n}:=\log y_{i,n}$ ($i=0,\dots,n-1$),
$X:=\log x$ and $Y:=\log y$.

In the following of this section, we shall fix $\sigma\in G_K$ and frequently
omit mentioning $\sigma$ that is potentially appearing 
in each term of our functional equation. 
In particular, the quantities $\chi$, $\rho_z$ designate 
the values $\chi(\sigma)$, $\rho_z(\sigma)$ at $\sigma\in G_K$ 
respectively.
For our fixed $\sigma\in G_K$, let us determine the polylogarithmic
part of the Galois transformation 
$\ff_\sigma^\gamma:=\gamma\cdot\sigma(\gamma)^{-1}$
of the path $\gamma:\ovec{01}_n\nyoroto z$ in the form:
\begin{align}
\label{old4.6}
\log(\ff_\sigma^\gamma)^{-1}
&\equiv CX_n+\sum_{s=0}^{n-1}\sum_{m=1}^\infty 
C_{s,m}\, \mathrm{ad}(X_n)^{m-1}(Y_{s,n})\\
&\equiv CX_n+\sum_{s=0}^{n-1}\CC_s(\mathrm{ad} X_n)(Y_{s,n})
\quad\text{ mod }
I_{Y_\ast},
\nonumber
\end{align}
where, $I_{Y_\ast}$ represents 
the ideal generated by those terms including 
$\{Y_{0,n},\dots,Y_{n-1,n}\}$ twice or more, and
$\CC_s(t)=\sum_{m=1}^\infty C_{s,m} t^{m-1}
\in\Q_\ell[[t]]$ $(s=0,\dots,n-1)$.

We determine the above coefficients $C$, $C_{s,m}$
by applying the morphisms $\mathj_\zeta$ ($\zeta\in\mu_n$).
Let us set
\begin{align*}
&\sLL^{(\zeta)}(t):=\sLL_1(\zeta z)+\sLL_2(\zeta z)t+\sLL_3(\zeta z)t^2+\cdots 
\quad (\zeta \in\mu_n); \\
&\sLL^{(n)}(t):=\sLL_1(z^n)+\sLL_2(z^n)t+\sLL_3(z^n)t^2+\cdots, 
\end{align*}
with
\begin{align*}
&\begin{cases}
&\sLL_0(\zeta z):= \rho_{\zeta z} =\rho_z+\frac{s}{n}(\chi-1) \\
&\sLL_1(\zeta z):= \rho_{1-\zeta z}, \\
&\sLL_k(\zeta z):= \frac{\tilchi_k^{\zeta z}(\sigma)}{(k-1)!}
\quad(k\ge 2); \\
&(\zeta=e^{2\pi i s/n},\ s=0,1,\dots,n-1 ),  \\
\end{cases}
\quad
\begin{cases}
\sLL_0(z^n)&:=\rho_{z^n}=n\rho_z, \\
\sLL_1(z^n)&:=\rho_{1-z^n}
=\sum_{\zeta\in\mu_n} \rho_{1-\zeta z},  \\
\sLL_k(z^n)&:=\frac{\tilchi^{z^n}_{k}(\sigma)}{(k-1)!}
\quad(k\ge 2).
\end{cases}
\end{align*}
Then, 

\begin{Lemma}
\label{determinCsm}
\begin{align*}
(1)\quad &C=\sLL_0(z)=\rho_z.
\\
(2)\quad
&\CC_0(t)=\sLL^{(1)}(-t)\frac{\rho_z  t}{e^{\rho_z t}-1}.
\\
(3)\quad &\CC_s(t)=\sLL^{(\zeta)}(-t)e^{((\frac{s}{n}-1)\chi-\frac{s}{n})t}
\frac{\rho_z  t}{e^{\rho_z t}-1}
\qquad
(s=1,\dots,n-1;\ \zeta=e^{-\frac{2\pi is}{n}}).
\end{align*}
\end{Lemma}

The proof of this lemma will be given later 
in this section.

\bigskip
\noindent
{\bf Proof of Theorem
\ref{MainFormulaInhomo} assuming 
Lemma \ref{determinCsm}:}

\medskip
We apply the morphism $\pi_n:V_n\to V_1$ to 
$\log(\ff_\sigma^{\gamma})^{-1}$.
We first observe that
$\pi_n(X_n)=nX$, 
$\pi_n(Y_{s,n})=x^sYx^{-s}=\sum_{k=0}^\infty
\frac{s^k}{k!}(\adX)^k(Y)
=e^{s\cdot\adX}(Y)$ 
for $s=0,\dots,n-1$.
Hence,
\begin{equation}
\label{old4.12}
\pi_n(\log(\ff_\sigma^{\gamma})^{-1})
=CnX+\sum_{s=0}^{n-1}
\CC_s(n\,\adX)
\left(
\sum_{k=0}^\infty
\frac{s^k}{k!}(\mathrm{ad} X)^k
\right)
(Y).
\end{equation}
The above LHS equals to
\begin{equation}
\label{old4.13}
\log(\ff_\sigma^{\pi_n(\gamma)})^{-1}
=
\rho_{z^n} X+\sum_{k=1}^\infty
\ll_k(z^n,\pi_n(\gamma))(\adX)^{k-1}(Y).
\end{equation}
{}From the formula (\ref{explicitNW1}) we see that
\begin{equation}
\label{old4.5}
\sum_{k=1}^\infty \ll_k(z^n,\pi_n(\gamma)) t^k
=
t\,\sLL^{(n)}(-t)\frac{\rho_z  nt}{e^{\rho_z  nt}-1},
\end{equation}
hence that the equality of  RHSs of 
(\ref{old4.12}) and (\ref{old4.13}) results in:
\begin{equation}
\label{old4.14}
\sum_{s=0}^{n-1}\CC_s(nt)e^{st}=
\sLL^{(n)}(-t)\frac{\rho_z  nt}{e^{\rho_z  nt}-1}.
\end{equation}

Substituting $\CC_s(t)$ $(s=0,\dots, n-1)$ 
by Lemma \ref{determinCsm} (2), (3),
the above left side equals 
\begin{equation}
\label{old4.15}
\left(
\sLL^{(1)}(-nt)+\sum_{s=1}^{n-1}
\sLL^{(\zeta)}(-nt)
e^{((\frac{s}{n}-1)\chi-\frac{s}{n})nt}e^{st}
\right)
\frac{\rho_z  nt}{e^{\rho_z nt}-1}
\end{equation}
where, in the summation $\sum_s$,  we understand
$\zeta=e^{-\frac{2\pi is}{n}}$.
As 
$((\frac{s}{n}-1)\chi-\frac{s}{n})nt+st=-(n-s)\chi t$, 
the replacement of $\zeta$ by 
$\zeta_n^{\,s}=e^{\frac{2\pi is}{n}}$
enables us to collect the sum as  
$\sum_{s=0}^{n-1}\sLL^{(\zeta)}(-nt)e^{-s\chi t}$.
Finally, substituting $t$ for $-t$, we obtain
\begin{equation}
\label{old4.16}
\sLL^{(n)}(t)=
\sum_{s=0}^{n-1}
\sLL^{(\zeta)}(nt)e^{s\chi t}
\qquad
(\zeta=e^{\frac{2\pi is}{n}}).
\end{equation}
Theorem \ref{MainFormulaInhomo} 
follows from comparing the coefficients of
the above equation.
\qed

\medskip
We prepare the following combinatorial lemma
concerning the Baker-Campbell-Hausdorff sum:
$S\CHplus T=\log(e^Se^T)$.
Let 
$$
\beta(t)=\frac{t}{e^t-1}=\sum_{n=0}^\infty B_n\frac{t^n}{n!}
$$
be the generating function for Bernoulli numbers.

\begin{Lemma}
\label{BCH}
Let $K$ be a field of characteristic 0 and let 
$\alpha$, $\ell_0,\ell_1,\dots\in K$. 
Let $\ell(X,Y)=\ell_0X+\ell_+(\adX)(Y)
=\ell_0X+\sum_{k=1}^\infty
\ell_{k}\,(\adX)^{k-1}(Y)$
be an arbitrary element of 
the formal Lie series ring $\Lie_K\lala X,Y\rara$
with $\ell_+(t)\in K[\![t]\!]$.
Then, we have the following congruence formulas modulo 
$I_Y$.
{\small
\begin{align*}
(i) \quad \ell(X,Y) \CHplus \alpha X &\equiv
(\alpha+\ell_0)X+\left(
\frac{\beta((\alpha+\ell_0)\adX)}{\beta(\alpha\, \adX)}
\ell_+(\adX)\right)(Y);
\\
(ii)\quad
\alpha X \CHplus \ell(X,Y) &\equiv
(\alpha+\ell_0)X+\left(
\frac{\beta((\alpha+\ell_0)\adX)}{\beta(\ell_0\,\adX)}
\ell_+(\adX)e^{\alpha\,\adX}\right)(Y).
\end{align*}
}
\end{Lemma}

\begin{proof}
Both formulas follow from the polylogarithmic BCH formula and 
with a representation of the core generating function.
See \cite[Prop.\,5.9 and (5.8)]{NW2}.
\end{proof}

\bigskip
\noindent
{\bf Proof of Lemma \ref{determinCsm}:}
\medskip

Apply
the morphisms $\jmath_\zeta$ ($\zeta\in\mu_n$)
to determine the coefficients $C_{m,s}$ 
of the polylogarithmic terms of 
$\log\ff_\sigma^\gamma$ in (\ref{old4.6}).

\medskip
\noindent
{\bf Case $\zeta=1$}:
Observe that $\jmath_1(X_n)=X$, $\jmath_1(Y_{0,n})=Y$ and 
$\jmath_1(Y_{i,n})=0$ $(i\ne 0)$.
Then, it follows from (\ref{old4.6}) that
$$
\jmath_1(\log(\ff_\sigma^\gamma)^{-1})
=\log(\ff_\sigma^{\jmath_1(\gamma)})^{-1}
\equiv
CX+
\bigl(\CC_0(\adX)\bigr)(Y)
\quad\text{ mod }I_Y.
$$
We immediately see that the first coefficient $C$ is
given by
\begin{equation}
\label{old4.8}
C=\rho_z=\sLL_0(z),
\end{equation}
and that the other polylogarithmic coefficients
are given by (\ref{explicitNW1}) as follows:
\begin{equation}
\label{old4.10}
\CC_0(t)=
\sum_{k=1}^\infty \ell i_k(z,\mathj_1(\gamma))t^{k-1}
=\sLL^{(1)}(-t)\frac{\rho_z  t}{e^{\rho_z t}-1}.
\end{equation}

\medskip
\noindent
{\bf Case $\zeta\ne 1$}:
Assume $\zeta=e^{-\frac{2\pi is}{n}}$ $(s=1,\dots,n-1)$.
We observe in this case that 
$\delta_\zeta\,\jmath_\zeta(X_n)\,\delta_\zeta^{-1}=X$, 
$\delta_\zeta\,\jmath_\zeta(Y_{s,n})\,\delta_\zeta^{-1}
=xYx^{-1}=\sum_{k=0}^\infty\frac{(\adX)^k(Y)}{k!}=e^{\adX}(Y)$ 
and $\jmath_\zeta(Y_{i,n})=0$ $(i\ne 0)$.
Therefore, it follows from (\ref{old4.6}) that
\begin{equation}
\label{eq5.9new}
\delta_\zeta\cdot
\jmath_\zeta(\log(\ff_\sigma^\gamma)^{-1})
\cdot \delta_\zeta^{-1} \equiv
CX+\bigl(\CC_s(\adX)e^{\adX}\bigr)(Y)
\quad\text{ mod }I_Y.
\end{equation}
On the other side, since 
$\ff_\sigma^{\delta_\zeta \mathj_\zeta(\gamma)}
=\delta_\zeta\ff_\sigma^{\mathj_\zeta(\gamma)}\delta_\zeta^{-1}
\ff_\sigma^{\delta_\zeta}$
by (\ref{def-of-ff}),
we have
%
\begin{align}
\label{old4.7}
\delta_\zeta &\cdot
\mathj_\zeta(\log(\ff_\sigma^\gamma)^{-1})
\cdot \delta_\zeta^{-1} 
=
\delta_\zeta\cdot
\log(\ff_\sigma^{\mathj_\zeta(\gamma)})^{-1}
\cdot \delta_\zeta^{-1}  \\
&=
\left(
-\log(\ff_\sigma^{\delta_\zeta})^{-1}
\right)
\CHplus
\left(
\log(\ff_\sigma^{\delta_\zeta \mathj_\zeta(\gamma)})^{-1}
\right)
\nonumber
\\
&\equiv
\left(-\frac{n-s}{n}(\chi-1)X\right)
\CHplus
\left(
\ell i_0(\zeta z)X+
\sum_{k=1}^\infty
\ell i_k(\zeta z)(\adX)^{k-1}(Y)
\right)
\nonumber
\end{align}
mod $I_Y$, where $\ell i_k(\zeta z)$ $(k\ge 0)$
are taken along the path
$\delta_\zeta \mathj_\zeta(\gamma)$.
Note here that 
$\ell i_0(\zeta z)=\sLL_0(\zeta z)=\rho_z+\frac{n-s}{n}(\chi-1)$
and that (\ref{explicitNW1}) implies
\begin{equation}
\label{li_k_zetaz}
\sum_{k=1}^\infty
\ell i_k(\zeta z, \delta_\zeta \mathj_\zeta(\gamma)) t^{k-1}
=
\sLL^{(\zeta)}(-t) \, \beta(\sLL_0(\zeta z)\,t)
=\sLL^{(\zeta)}(-t)\frac{\sLL_0(\zeta z)  t}{e^{\sLL_0(\zeta z) t}-1}.
\end{equation}
Putting this into (\ref{old4.7}) and using Lemma \ref{BCH}(ii), we find
\begin{equation}
\delta_\zeta\cdot
\jmath_\zeta(\log(\ff_\sigma^\gamma)^{-1})
\cdot \delta_\zeta^{-1} \equiv
\rho_zX+
\left(\sLL^{(\zeta)}(-\adX)
e^{-\frac{n-s}{n}(\chi-1) \adX}
\frac{\rho_z \adX}{e^{\rho_z \adX}-1}\right)(Y)
\end{equation}
mod $I_Y$.
Comparing this with (\ref{eq5.9new}),
we obtain
\begin{equation}
\label{old4.11}
\CC_s(t)=\sLL^{(\zeta)}(-t)e^{(-\frac{n-s}{n}(\chi-1)-1)t}
\frac{\rho_z  t}{e^{\rho_z t}-1}
\qquad
(s=1,\dots,n-1;\ \zeta=e^{-\frac{2\pi is}{n}}).
\end{equation}
Thus, the proof of Lemma \ref{determinCsm} is completed.
\qed

\begin{Remark}
In \cite[Theorem 5.7]{NW2}, we gave a general tensor 
criterion to have a functional equation of 
(complex and $\ell$-adic) polylogarithms
from a collection of morphisms 
$\{f_i:X\to \bold P^1-\{0,1,\infty\}\}_{i\in I}$ and 
their formal sum $\sum_{i\in I} c_i [f_i]$. 
In our above case, it holds that the collection 
$\{\pi_n, \mathj_0,\dots,\mathj_{n-1}:
V_n\to V_1\}$ satisfies the criterion 
with coefficients $1,-n^{k-1},\dots,-n^{k-1}$
(as observed already in \cite[(1.9) (iii)]{Ga}).
Explicit evaluation of the error terms
$\EE_k:=\EE_k(\sigma,\gamma)$ discussed 
in \cite{NW2} 
(that explains part of lower degree inhomogeneous terms
of our functional equation)  
can be obtained a posteriori 
from (\ref{old4.5}), (\ref{old4.10}), (\ref{li_k_zetaz}) 
and (\ref{old4.16}) as:
\begin{equation*}
\sum_{k=1}^\infty \EE_k t^k
=
\frac{\rho_z nt^2}{e^{\rho_z nt}-1}
\sum_{s=1}^{n-1}
\sLL^{(\zeta_n^s)}(-nt)
(e^{-s\chi t}-e^{-s(\chi-1) t}).
\end{equation*}
Note that the lower degree terms other than $\EE_k$
are explained by the Roger type normalization 
(difference from $\ll_k$ and $\tilchi_k$)
and the effects from compositions of paths 
$\ovec{01}\nyoroto\zeta\ovec{01}\nyoroto \zeta z$
of Baker-Campbell-Hausdorff type.
\end{Remark}

\begin{Remark}
Replacing $\sLL^{(n)}(t)$, $\sLL^{(\zeta)}(nt)$ 
in (\ref{old4.16})
by those generating functions 
for 
\linebreak
$\ell i_k(z^n, \pi_n(\gamma))$, 
$\ell i_k(\zeta z, \delta_\zeta\mathj_\zeta(\gamma))$
by (\ref{old4.5}), (\ref{old4.10}) and (\ref{li_k_zetaz}),
we obtain an equation
\begin{align*}
&\sum_{k=1}^\infty
\ell i_k(z^n,\pi_n(\gamma)) t^{k-1} \\
&=
\frac{\rho_z nt}{e^{\rho_z nt}-1}
\sum_{s=0}^{n-1}
e^{s\chi t}
\left(
\frac{e^{-\sLL_0(\zeta_n^s z)nt}-1}{-\sLL_0(\zeta_n^s z)nt}
\right)
\sum_{k=1}^\infty
\ell i_k(\zeta_n^s z, \mathj_{\zeta_n^s}(\gamma)) (-nt)^{k-1}
\end{align*}
in $\Q_\ell[\![t]\!]$.
{}From this, for every fixed $k\ge 1$, one may express 
$\ell i_k(z^n,\pi_n(\gamma)) $
as a linear combination of the
$\ell i_d(\zeta_n^s z, \delta_{\zeta_n^s}\mathj_{\zeta_n^s}(\gamma))$
($s=0,\dots,n-1$, $d=1,\dots,k$).
However, those coefficients are apparently more complicated
than those in Theorem \ref{MainFormulaInhomo} 
where the polylogarithmic characters 
$\tilchi_k^{z^n}$, $\tilchi_d^{\zeta z}$ are treated.
\end{Remark}

\section{Homogeneous form}

We keep the notations in \S 5 with assuming
$\mu_n\subset K$.
Let $\pi_{\Q_\ell}(\ovec{01}_n)$ denote
the {\it $\ell$-adic pro-unipotent fundamental group}
of $V_n\otimes\bK$ based at $\ovec{01}_n$
which is by definition the pro-unipotent
hull of the image of the 
Magnus embedding (\ref{MagnusEmbedding})
consisting of all the group-like elements of the
complete Hopf algebra
$\Q_\ell\lala X_n, Y_{i,n}\mid 0\le i \le n-1 \rara$.
We also define the 
{\it $\ell$-adic pro-unipotent path space}
(or $\Q_\ell$-{\it path space} for short)
$\pi_{\Q_\ell}(\ovec{01}_n,v)$ for a 
$K$-(tangential) point $v$ on $V_n$
to be the $\Q_\ell$-rational extension of
the path torsor 
$\pi_1^\ell(V_n\otimes\bK,\ovec{01}_n,v)$ via 
$\pi_1^\ell(V_n\otimes\bK,\ovec{01}_n)
\subset\pi_{\Q_\ell}(\ovec{01}_n)$.
Note that both 
$\pi_{\Q_\ell}(\ovec{01}_n)$ and 
$\pi_{\Q_\ell}(\ovec{01}_n,v)$
have natural actions by $G_K$
compatible with identification
$$
\pi_1^\ell(V_n\otimes\bK,\ovec{01}_n)
\subset\pi_{\Q_\ell}(\ovec{01}_n),\quad
\pi_1^\ell(V_n\otimes\bK,\ovec{01}_n,v)
\subset\pi_{\Q_\ell}(\ovec{01}_n,v).
$$
Let us introduce rational modifications of the loops $y_{s,n}$
($s=0,\dots, n-1$) and the paths 
$\delta_\zeta$ ($\zeta\in\mu_n$) respectively as follows.
For $s=0,\dots,n-1$ and $\zeta=e^{2\pi is/n}$, set
\begin{align*}
\cy_{s,n}&:=x_n^{-\frac{s}{n}} y_{s,n} x_n^{\frac{s}{n}}\in 
\pi_{\Q_\ell}(\ovec{01}_n),\\
\eps_\zeta&:= x^{-\frac{s}{n}}\cdot \delta_\zeta\in 
\pi_{\Q_\ell}(\ovec{01},\zeta\ovec{01}).
\end{align*}
Note that, in the case $n=1$, we have
$x=x_1$, $y=\tilde{y}_{0,1}$ by definition.

\medskip
The following lemma is the key to homogenize the $\ell$-adic
distribution formula.

\begin{Lemma}
\label{homogeneousLemma}
(i) For every $\sigma\in G_K$ and $\zeta\in\mu_n$, 
we have $\sigma(\eps_\zeta)=\eps_\zeta$. 
Moreover, for any path 
$\gamma$ from $\zeta\ovec{01}$ to a $K$-point
$w$ on $V_1$, 
we have $\eps_\zeta\,\ff_\sigma^{\,\gamma}
\, \eps_\zeta^{-1}
=\ff_\sigma^{\eps_\zeta\, \gamma}$.

(ii) The natural extensions of the homomorphisms on
$\pi_{\Q_\ell}(\ovec{01}_\ast)$
induced by $\jmath_\zeta: V_n\to V_1$, 
$\pi_{rn,r}:V_{rn}\to V_r$ (denoted by the same symbols)
map the loops
$x_n$, $\cy_{s,n}$ $(s=0,\dots,n-1)$ as follows.
\begin{align*}
(a)\quad 
&\eps_\zeta\,\jmath_\zeta(x_n)\,\eps_\zeta^{-1}=x. 
\\
(b)\quad
&\pi_{rn,r}(x_{rn})=x_r{}^n.  
\\
(c)\quad
&\eps_\zeta\,\jmath_\zeta(\cy_{s,n})\,\eps_\zeta^{-1}=
\begin{cases}
y &(\zeta=e^{-2\pi is/n}), \\
1 &(\zeta\ne e^{-2\pi is/n}).
\end{cases} 
\\
(d)\quad
&\pi_{rn,r}(\cy_{j,rn})= \cy_{i,r}  
\qquad(0\le i<r,\ 0\le j<rn,\ i\equiv j \ \mathrm{ mod }\ r).
\end{align*}
\end{Lemma}

\noindent
\begin{proof}
(i): Let $\zeta=e^{2\pi is/n}$ ($s=0,\dots,n-1$). 
By the assumption
$\mu_n\subset K$, we have 
$\chi(\sigma)\equiv 1$ mod $n$ for $\sigma\in G_K$.
The first assertion follows immediately from the formula 
$$
\sigma(\delta_\zeta)=x^{\frac{s}{n}(\chi(\sigma)-1)}\delta_\zeta,
$$ 
which can be easily seen from an argument similar to the proof
of \cite[Prop.\,1]{NW1} with
replacement of $\bar F(\!(t-z)\!)$ by $\bar F\{\!\{\zeta t\}\!\}$.
The second claim follows easily from 
the definition (\ref{def-of-ff}):
$\ff_\sigma^p=p\cdot\sigma(p)^{-1}$ for any path $p:a\nyoroto b$.

(ii): (a), (b) and the case $\zeta\ne e^{-2\pi is/n}$ of (c) are
trivial. (d) follows from (b) and the fact 
$\pi_{rn,r}(y_{j,rn})=x_r^ky_{i,r}x_r^{-k}$ with $j=i+kr$, $0\le i<r$
(\ref{eq3.2}).
It remains to prove (c) in the case $\zeta=e^{-2\pi is/n}$.
Suppose first that $\zeta$ is different from $1$, i.e., 
$\zeta=e^{-2\pi is/n}$ for any fixed $s=1\dots n-1$.
Then $\eps_\zeta=x^{-\frac{n-s}{n}}\cdot\delta_\zeta$.
Since
$\delta_\zeta\,\jmath_\zeta(y_{s,n})\,\delta_\zeta^{-1}=xyx^{-1}$,
(a) implies 
$\delta_\zeta \jmath_\zeta(\cy_{s,n})\delta_\zeta^{-1}=
x^{-\frac{s}{n}}xyx^{-1}x^{\frac{s}{n}}
=x^{\frac{n-s}{n}}yx^{-\frac{n-s}{n}}$.
It follows then that 
$\eps_\zeta\jmath_\zeta(\cy_{s,n})\eps_\zeta^{-1}=y$.
Next, suppose $\zeta=1$ (i.e., $s=0$). Then, it is easy to settle
this case by $\jmath_1(y_{0,n})=y$.
We thus complete the proof of (c).
\end{proof}

Now, we embed $\pi_{\Q_\ell}(\ovec{01}_n)$ and its Lie algebra 
$L_{\Q_\ell}(\ovec{01}_n)$ into the non-commutative power
series ring $\Q_\ell\lala \cX_n, \cY_{s,n}\mid 0\le s<n\rara$
by setting $\cX_n:=X_n=\log x_n$, 
$\cY_{s,n}:=\log\cy_{s,n}$,
and denote by $\scM_n$ the set of monomials in 
$\cX_n, \cY_{s,n}$ $(s=0,\dots,n-1)$.
For $w\in \scM_n$, let $\wt_X(w)$ denote the number
of $\cX_n$ appearing in $w$.
We shall also employ the monomial congruence 
`$w\equiv w'$ mod $r$' by following the same manner
as Definition \ref{def2.2} after replacing $X_n$, $Y_{i,n}$
by $\cX_n$, $\cY_{i,n}$ ($n\in\Z_{>0}$, $0\le i<n$)
respectively.
For the case $n=1$, we will also simply write 
$\cX=\cX_1$, $\cY=\cY_{0,1}$.

\begin{Definition}
Let $z$ be a point in $V_n(K)$.
Given a $\Q_\ell$-path $p\in \pi_{\Q_\ell}(\ovec{01},z)$ 
and any $\sigma\in G_K$, we set 
$\ff_\sigma^p:=p\cdot\sigma(p)^{-1}$
and expand it in the form
\begin{equation*}
\label{scLiDef}
\ff_\sigma^p
=1+\sum_{w\in\scM_n}
\scLi_w(\pathto{\ovec{01}}{p}{z})(\sigma) \cdot w
\end{equation*}
in 
$\Q_\ell\lala \cX_n, \cY_{i,n}\mid 0\le i \le n-1 \rara$.
(Recall that, in (\ref{LiDef}), another (non-commutative) expansion 
of $\ff_\sigma^\gamma$ for 
$\gamma\in\pi_1^\ell(V_n\otimes\bK,\ovec{01},z)$
was considered by using a different set of variables.) 
We call the above coefficient character
$$
\scLi_w(\pathto{\ovec{01}}{p}{z})
\left(=\scLi_w^{(\ell)}(\pathto{\ovec{01}}{p}{z})
\right)
:G_K\to\Q_\ell
$$
the {\it $\ell$-adic iterated integral}
associated to the word $w\in \scM_n$ and to the
$\Q_\ell$-path $p$ on $V_n$.
\end{Definition}

\begin{Theorem}
\label{l-adicGeneralFormula}
Let $p$ be a $\Q_\ell$-path
on $V_{rn}$ from $\ovec{01}$ to a point $z\in V_{rn}(K)$.
Then, for any word $w\in\scM_r$, 
we have the distribution relation
$$
\scLi_w(\pathtoT{\ovec{01}}{\pi_{rn,r}(p)}{z^n})(\sigma)
=
n^{\wt_X(w)}
\sum_{\substack{u\in\scM_{rn}\\ u\equiv w\ \mathrm{mod}\ r}}
\scLi_u(\pathto{\ovec{01}}{p}{z})(\sigma)
$$
for $\sigma\in G_K$.
\end{Theorem}

\begin{proof}
The assertion follows in the same way as Theorem \ref{complexGeneralFormula}
after the above Lemma \ref{homogeneousLemma} (b), (d).
\end{proof}

Next, let us concentrate on the polylogarithmic part on $V_1$.
Recall that both $\pi_{\Q_\ell}(\ovec{01})$ and its Lie algebra $L_{\Q_\ell}(\ovec{01})$
are embedded in $\Q_\ell\lala X,Y\rara$, 
where $X=\cX_1$ and $Y=\cY_{0,1}$.

\begin{Definition}
\label{defofQellPolylog}
Let $z$ be a point in $V_1(K)=\bold P^1(K)-\{0,1,\infty\}$ and
$p:\ovec{01}\nyoroto z$ a $\Q_\ell$-path.
Consider the associator 
$\ff_\sigma^p:=p\cdot \sigma(p)^{-1}\in\pi_{\Q_\ell}(\ovec{01})$
for $\sigma\in G_K$, and 
define 
$$
\brho_{z,p}:G_K\to\Q_\ell, \quad \ll_m(z,p):G_K\to\Q_\ell
$$ 
by 
the non-commutative expansion corresponding to (\ref{ladicpolydef}):
$$
\log(\ff_\sigma^p)^{-1}
\equiv\brho_{z,p}(\sigma) X +
\sum_{m=1}^\infty
\ll_m(z,p)(\sigma)
(\mathrm{ad}X)^{m-1}(Y)
\quad\text{ mod }I_{Y},
$$
where $I_{Y}$ represents 
the ideal generated by those terms including 
$Y$ twice or more.
Using these, we also define
$$
\tilbchi_m^{z,p}:G_K\to\Q_\ell
$$
for $m\ge 1$ by the equation extending 
Proposition \ref{prop4.1} (i):
\begin{equation}
\tilbchi_m^{z,p}(\sigma)
=(-1)^{m+1}(m-1)!
\sum_{k=1}^m
\frac{{\brho_{z,p}(\sigma)}^{m-k}}{(m+1-k)!}
\ll_k(z, p)(\sigma).
\end{equation}
\end{Definition}

Since $\Q_\ell$-paths generally do not give
bijection systems between fibers of endpoints
on finite \'etale covers, no simple
interpretation is available for $\brho_{z,p}$ or  
$\tilbchi_m^{z,p}$ by Kummer properties:
For example, the above 
$\tilbchi_m^{z,p}(\sigma)$ $(\sigma\in G_K)$
generally has a denominator in $\Q_\ell$, i.e., 
may not be valued in $\Z_\ell$. 
This makes it difficult to understand $\tilbchi_m^{z,p}(\sigma)$ 
in terms of Kummer properties at finite levels 
of an arithmetic sequence like (\ref{defpolychar}).

Once $\brho_{z,p}$, $\ll_m(z,p)$ and 
$\tilbchi_m^{z,p}:G_K\to\Q_\ell$ are defined 
as in the above Definition, the identities as
in Proposition \ref{prop4.1} (ii) and (\ref{explicitNW1}) 
can be extended in obvious ways for them 
by formal transformations of generating functions.
In the same way, it holds that
\begin{equation}
-\frac{\tilbchi_m^{z,p}(\sigma)}{(m-1)!}=
\scLi_{YX^{m-1}}(\pathtoD{\ovec{01}}{p}{z})(\sigma)
\end{equation}
for $p\in\pi_{\Q_\ell}(\ovec{01},z)$ and $\sigma\in G_K$.

\begin{Theorem}
\label{l-adicpolythm}
Suppose $\mu_n\subset K\subset \C$ and let
$p$ be a $\Q_\ell$-path
on $V_{n}$ from $\ovec{01}$ to a point $z\in V_{n}(K)$.
Then,
$$
\ll_k(z^n, \pi_n(p))(\sigma)
=n^{k-1}\sum_{\zeta\in\mu_n}
\ll_k(\zeta z, \eps_\zeta \jmath_\zeta(p))(\sigma)
$$
holds for $\sigma\in G_K$.
\end{Theorem}

\noindent
{\it Proof.} We first put the Lie expansion of 
$\log (\ff_\sigma^p)^{-1}$ in 
$\cX_n=X_n=\log x_n$, $\cY_{s,n}=\log\cy_{s,n}$ ($s=0,\dots,n-1$) in
the Lie algebra $L_{\Q_\ell}(\ovec{01}_n)$
as:
\begin{align}
\log(\ff_\sigma^\gamma)^{-1}
&\equiv DX_n+\sum_{s=0}^{n-1}\sum_{m=1}^\infty 
D_{s,m}\, (\mathrm{ad}\cX_n)^{m-1}(\cY_{s,n})\\
&\equiv DX_n+\sum_{s=0}^{n-1}\left(\DD_s(\mathrm{ad} \cX_n)\right)(\cY_{s,n})
\quad\text{ mod }
I_{\cY_\ast},
\nonumber
\end{align}
where, $I_{\cY_\ast}$ represents 
the ideal generated by those terms including 
$\{\cY_{0,n},\dots,\cY_{n-1,n}\}$ twice or more, and
$\DD_s(t)=\sum_{m=1}^\infty D_{s,m} \,t^{m-1}
\in\Q_\ell[[t]]$ $(s=0,\dots,n-1)$.
We shall determine those coefficients $D$ and $\DD_{s,m}$ by
applying the morphisms $\jmath_\zeta$.
For any fixed $\zeta=\zeta_n^{-s}$ ($s=0,\dots,n-1$),
by Lemma \ref{homogeneousLemma} (i), we obtain
$\ff_\sigma^{\eps_\zeta \jmath_\zeta(p)}
=\eps_\zeta\cdot 
\jmath_\zeta(p\cdot\sigma(p)^{-1}) 
\cdot\sigma(\eps_\zeta)^{-1}
=\eps_\zeta\cdot 
\jmath_\zeta(\ff_\sigma^p) 
\cdot\eps_\zeta^{-1}
$, 
hence
$$
\eps_\zeta
\cdot \jmath_\zeta(\log(\ff_\sigma^p)^{-1})\cdot
\eps_\zeta^{-1}
=
\log(\ff_\sigma^{\eps_\zeta \jmath_\zeta(p)})^{-1}.
$$
As the right hand side comes from the associator for the
path $\eps_\zeta \jmath_\zeta(p):\ovec{01}\nyoroto \zeta z$,
it should coincide, by definition, with
$$
\brho_{\zeta z,  \eps_\zeta \jmath_\zeta(p)}(\sigma) X+
\sum_{k=1}^\infty
\ll_k(\zeta z, \eps_\zeta \jmath_\zeta(p))(\sigma)
\, (\mathrm{ad}X)^{k-1}(Y),
$$
while, the left hand side can be calculated after 
Lemma \ref{homogeneousLemma} (ii) (a), (c) to equal to
$$
D X+ 
\sum_{k=1}^\infty D_{s,k}\,(\adX)^{k-1}(Y)
$$
with $s$ given by $\zeta=e^{-2\pi is/n}$.
Therefore, we conclude
\begin{align}
D&=\brho_{\zeta z,  \eps_\zeta \jmath_\zeta(p)}(\sigma), 
\label{6.2eq}\\
D_{s,k}&=\ll_k(\zeta z, \eps_\zeta \jmath_\zeta(p))(\sigma)
\label{6.3eq}
\end{align}
for $\zeta=e^{-2\pi is/n}$ ($s=0,\dots,n-1)$.
Now, apply the projection morphism $\pi_n:=\pi_{n,1}:V_n\to V_1$
and interpret the both sides of equality 
$\pi_n(\log(\ff_\sigma^p)^{-1})=
\log(\ff_\sigma^{\pi_n(p)})^{-1}$.
Then, we obtain
$$
DnX+\sum_{s=0}^{n-1} 
\sum_{k=1}^\infty D_{s,k}\,(n\,\adX)^{k-1}(Y)
=
\brho_{z^n,\pi_n(p)}X+\sum_{k=1}^\infty
\ll_k(z^n, \pi_n(p))
\, \mathrm{ad}(X)^{k-1}(Y).
$$
Comparing the coefficient of 
$(\adX)^{k-1}(Y)$ in the above and (\ref{6.3eq}), we
conclude the proof of the theorem. \qed

\bigskip
In the above proof, for a given $\Q_\ell$-path 
$p:\ovec{01}\nyoroto z$ on $V_n$, we considered 
the collection of $\Q_\ell$-paths
$$
\mathscr{P}_n:=\{\eps_\zeta \jmath_\zeta(p):
\ovec{01}\nyoroto \zeta z
\mid \zeta=\zeta_n^s\in\mu_n
\ (s=0,1,\dots,n-1)
\}
$$ on $V_1=\mathbf{P}^1-\{0,1,\infty\}$ .
Note that each $\eps_\zeta \jmath_\zeta(p)$
can also be written as the composite of paths on $V_1$:
%
\begin{align}
&\eps_\zeta\cdot [\zeta p]=
x^{-\frac{s}{n}}\cdot
\delta_\zeta\cdot [\zeta p]
: \\
&\ovec{01}\overset{x^{-\frac{s}{n}}}{\leadsto\!\leadsto\ }
\ovec{01}\overset{\delta_\zeta}{\leadsto}
\zeta\ovec{01}\overset{[\zeta p]}{\leadsto\!\leadsto} 
\zeta z
\notag
\end{align}
where $[\zeta p]:\zeta\ovec{01}\nyoroto \zeta z$ means
a path obtained by ``rotating'' 
$p:\ovec{01}\nyoroto z$ by the automorphism of 
$\mathbf{P}^1-\{0,\infty\}$ with multiplication by $\zeta$.

\begin{Corollary}
\label{brhoinvariance}
Notations being as above, 
the maps $\brho_{\zeta z,p}(\sigma):G_K\to\Q_\ell$ are all the
same for the $\Q_\ell$-paths
$[p:\ovec{01}\nyoroto \zeta z]
\in \bigcup_{n=1}^\infty\mathscr{P}_n$. 
\end{Corollary}

\begin{proof}
As seen in (\ref{6.2eq}),
we have the common $D$ 
upon applying $\jmath_\zeta$ to the first term of
$\log(\ff_\sigma^p)^{-1}$.
The assertion follows from this and the fact that
$\mathscr{P}_n$ contains $\mathscr{P}_1\ne\emptyset$.
\end{proof}

\noindent
{}From this corollary, we immediately see that the above theorem also gives 
homogeneous functional equations for the rationally extended 
$\ell$-adic polylogarithmic characters.
\begin{Corollary}
\label{tilchi-homogeneous}
Notations being as in Theorem \ref{l-adicpolythm},
let $\tilbchi_k^{z^n, \pi_n(p)}$ and 
$\tilbchi_k^{\zeta z, \eps_\zeta \jmath_\zeta(p)}$ $(\zeta\in\mu_n)$
be the extended
$\ell$-adic polylogarithmic characters. 
Then, we have
$$
\tilbchi_k^{z^n,\pi_n(p)} (\sigma)
=n^{k-1}\sum_{\zeta\in\mu_n}
\tilbchi_k^{\zeta z, \eps_\zeta \jmath_\zeta(p)}(\sigma)
\qquad (\sigma\in G_K).
$$
\end{Corollary}

\begin{proof}
The assertion follows from Theorem \ref{l-adicpolythm}
by applying Corollary \ref{brhoinvariance} to
the definition of 
$\ell$-adic polylogarithmic characters
for $\Q_\ell$-paths (Definition \ref{defofQellPolylog}).
\end{proof}

\section{Translation in Kummer-Heisenberg measure}

Let $\gamma:\ovec{01}\nyoroto z$ be an $\ell$-adic 
path in $\pi_1^\ell(\bold P_{\bK}^1-\{0,1,\infty\};\ovec{01},z)$ and 
$p:=x^{-\frac{s}{n}}\gamma$ be the pro-unipotent path in
$\pi_{\Q_\ell}(\ovec{01},z)$ produced by the composition
with $x^{-\frac{s}{n}}$ for any fixed $s\in\Z_\ell$ and $n\in \mathbb{N}$.
By definition we have 
$\ff_\sigma^p=
x^{-\frac{s}{n}}\ff_\sigma^\gamma x^{\frac{s}{n}\chi(\sigma)}$
for $\sigma\in G_K$. 
Since $x^{-\frac{s}{n}}=\exp(\frac{-s}{n}X)\equiv 1$
modulo the right ideal $X\!\cdot\!\Q_\ell\lala X,Y\rara$, it follows from
Proposition \ref{prop4.1} (ii)
that
\begin{align*}
-\frac{\tilbchi_k^{z,p}(\sigma)}{(k-1)!}
&=\mathsf{Coeff}_{YX^{k-1}}(\ff_\sigma^p) 
=\mathsf{Coeff}_{YX^{k-1}}
\left(1\cdot
\ff_\sigma^\gamma
\cdot\exp\biggl(\frac{s\,\chi(\sigma)}{n}X\biggr)
\right) \\
&=
\sum_{i=0}^{k-1}
\mathsf{Coeff}_{YX^{i}}(\ff_\sigma^\gamma)\cdot
\frac{\left(\frac{s}{n}\chi(\sigma)\right)^{k-i-1}}{(k-i-1)!} 
=-\sum_{i=0}^{k-1}
\frac{\tilchi_{i+1}^{z,\gamma}(\sigma)}{i\,!}\cdot
\frac{\left(\frac{s}{n}\chi(\sigma)\right)^{k-i-1}}{(k-i-1)!} .
\end{align*}
Thus we obtain
\begin{equation}
\label{eq7.1}
\tilbchi_k^{z,p}(\sigma)=
\sum_{i=0}^{k-1}\binom{k-1}{i}
\left(\frac{s}{n}\chi(\sigma)\right)^{k-i-1}
\tilchi_{i+1}^{z,\gamma}(\sigma)
\qquad(\sigma\in G_K).
\end{equation}
Recall then that, 
in \cite{NW1}, introduced is a certain $\Z_\ell$-valued measure
(called the Kummer-Heisenberg measure)
$\bkappa_{z,\gamma}(\sigma)$ on $\Z_\ell$
for every path $\gamma:\ovec{01}\nyoroto z$
and $\sigma\in G_K$, which is characterized by
the integration properties:
\begin{equation}
\label{NW1kummer}
\tilchi_k^{z,\gamma}(\sigma)
=\int_{\Z_\ell}a^{k-1}
d\bkappa_{z,\gamma}(\sigma)(a)
\qquad (k\ge 1).
\end{equation}
Putting this into (\ref{eq7.1}), we may rewrite the RHS
to get
\begin{equation}
\label{eq7.3}
\tilbchi_k^{z,p}(\sigma)=
\int_{\Z_\ell}\biggl(a +\frac{s}{n}\chi(\sigma)\bigg)^{k-1}
d\bkappa_{z,\gamma}(\sigma)(a).
\end{equation}
Note that 
$\frac{s}{n}+\Z_\ell=\frac{s}{n}\chi(\sigma)+\Z_\ell$
as a subset of $\Q_\ell$
when $\mu_n\subset K$. 
Comparison of (\ref{NW1kummer}) and (\ref{eq7.3})
leads us to introduce the following

\begin{Definition}
\label{defbkappa}
Suppose $\mu_n\subset K$, and
let 
$\sigma\in G_K$ and 
$p=x^{-\frac{s}{n}}\gamma\in\pi_{\Q_\ell}(\ovec{01},z)$
be as above. 
Define a $\Z_\ell$-valued measure 
$\bkappa_{z,p}(\sigma)$ 
on the coset $\frac{s}{n}+\Z_\ell(\subset\Q_\ell)$ 
by the property:
\begin{equation*}
\tilbchi_k^{z,p}(\sigma)=
\int_{\frac{s}{n}+\Z_\ell}a^{k-1}
d\bkappa_{z,p}(\sigma)(a)
\qquad (k\ge 1).
\end{equation*}
\end{Definition}

A verification of  
this new notion of the extended measure 
$\bkappa_{z,p}(\sigma)$
is that our distribution relations in 
Corollary \ref{tilchi-homogeneous}
can be summarized into a single relation
of measures:

\begin{Theorem}
For $s\in \Z_\ell$, let 
$[n]:\frac{s}{n}+\Z_\ell\to \Z_\ell$ $(a\mapsto na)$
denote the continuous map of multiplication by $n\in\mathbb{N}$,
and denote by 
$[n]_\ast \bkappa$
the push-forward measure on $\Z_\ell$ 
obtained from any measure $\bkappa$ on $\frac{s}{n}+\Z_\ell$
by
$U\mapsto \bkappa([n]^{-1}(U))$
for the compact open subsets $U$ of $\Z_\ell$.
Then, 
\begin{equation*}
\bkappa_{z^n,\pi_n(\gamma)}(\sigma)
=
\sum_{\zeta\in\mu_n}
[n]_\ast \bkappa_{\zeta z, \eps_\zeta\mathj_\zeta(\gamma)}(\sigma)
\qquad (\sigma\in G_K). 
\end{equation*}
\end{Theorem}

\begin{proof}
The formula follows immediately from Corollary \ref{tilchi-homogeneous}
and the characteristic property (\ref{eq7.3})
of the Kummer-Heisenberg measure.
\end{proof}

\begin{Question}
In the above discussion, we defined $\bkappa_{z,p}(\sigma)$
only for $\Q_\ell$-paths $p:\ovec{01}\nyoroto z$ of the form
$p=x^{\alpha}\gamma$ with $\alpha\in\Q_\ell$ and 
$\gamma:\ovec{01}\nyoroto z$ 
being $\ell$-adic (i.e., $\Z_\ell$-integral) paths.
It is natural to conjecture existence of a suitable measure 
$\bkappa_{z,p}(\sigma)$ for a more general $\Q_\ell$-path $p:\ovec{01}\nyoroto z$
satisfying the property of Definition \ref{defbkappa}.
The support of this measure 
should be a parallel transport $R(p,\sigma)$ of $\Z_\ell$ in $\Q_\ell$
such that $x^{R(p,\sigma)}\subset x^{\Q_\ell}$ is the image
of $\pi_1^\ell(V_1\otimes \overline{K};\ovec{01},z)\cdot \sigma(p)^{-1}$
via the projection $\pi_{\Q_\ell}(\ovec{01}) \twoheadrightarrow x^{\Q_\ell}$.
\end{Question}

\section{Inspection of special cases}

In this section, we shall closely look at special cases of 
the $\ell$-adic distribution formula.
Let us first consider dilogarithms, i.e., for the case of $k=2$.
By Theorem \ref{MainFormulaInhomo}, we have
\begin{Corollary}
Let $\mu_n\subset K$ and $\gamma:\ovec{01}\nyoroto z\in 
V_n(K)$ be an $\ell$-adic path which induces paths
$\pi_n(\gamma):\ovec{01}\nyoroto z^n$ and 
$\delta_\zeta\mathj_\zeta(\gamma):\ovec{01}\nyoroto \zeta z$
($\zeta=\zeta_n^s\in\mu_n$) on $V_1=\bold P^1-\{0,1,\infty\}$. 
Along these paths, we have the following $\Z_\ell$-valued
functional equation
$$
\tilchi_2^{z^n}(\sigma)
=
n\sum_{s=0}^{n-1}\tilchi_2^{\zeta_n^{\,s} z}(\sigma)
+\sum_{s=1}^{n-1} s \chi(\sigma)\rho_{1-\zeta_n^{\,s}z}(\sigma)
\qquad(\sigma\in G_K),
$$
where $\rho_{1-\zeta_n^{\,s}z}$ is the same as
the 1st polylogarithmic character $\tilchi_1^{\zeta_n^{\,s} z}:G_K\to \Z_\ell$.
\qed
\end{Corollary}

In particular when $n=2$, the above formula is specialized to
the following. 
\begin{Corollary}
\label{coro7.2}
For $\gamma:\ovec{01}\nyoroto z$ on $V_2=\bold P^1-\{0,\pm 1,\infty\}$,
let $\pi_2(\gamma):\ovec{01}\nyoroto z^2$,
$\mathj_1(\gamma):\ovec{01}\nyoroto z$ and
$\delta_{-1}\mathj_{-1}(\gamma):\ovec{01}\nyoroto -z$
be the induced paths on $\bold P^1-\{0,1,\infty\}$.
Note here that 
$\delta_{-1}:\ovec{01}\nyoroto-\ovec{01}$ is the positive
half rotation.
Along these paths, we have a functional equation of
the $\ell$-adic polylogarithmic characters
$$
\tilchi_2^{z^2}(\sigma)
=
2(\tilchi_2^z(\sigma)+\tilchi_2^{-z}(\sigma))
+\chi(\sigma)\rho_{1+z}(\sigma)
\qquad (\sigma\in G_K).
\qed
$$
\end{Corollary}

\noindent
Putting $z=\ovec{10}$ in the above, and recalling
$\tilchi_{2k}^{\ovec{10}}(\sigma)=
\frac{B_{2k}}{2(2k)}(\chi(\sigma)^{2k}-1)$
$(\sigma\in G_\Q)$ from [NW2] Proposition 5.13,
we immediately obtain

\begin{Corollary}
\label{dilogspecial}
Along the path $\gamma_{-1}:\ovec{01}\nyoroto (z\!=\!1) \nyoroto (z\!=\!-1)$ 
induced by the positive half arc on the unit circle on 
$\bold P^1-\{0,1,\infty\}$, we have
the following $\Z_\ell$-valued equation: 
$$
\tilchi_2^{z=-1}(\sigma)=-\frac{\chi(\sigma)^2-1}{48}
-\frac{1}{2}\chi(\sigma)\rho_2(\sigma)
\qquad (\sigma\in G_\Q).
\qed
$$
\end{Corollary}

\noindent
This result is an $\ell$-adic analog of 
the classical result 
$Li_2(-1)=-\frac{\pi^2}{12}$ ([Le]), and is 
compatible with 
\cite[Remark 5.14 and Remark after (6.31)]{NW2}.

\medskip
To confirm validity of our above narrow stream 
of geometrical arguments toward Corollary \ref{dilogspecial},
we here present an alternative direct proof 
in a purely arithmetic way as below:

\medskip
{\bf Arithmetic proof of Corollary \ref{dilogspecial}}.
\nolinebreak
We (only) make use of the characterization of
$\tilchi_m^z$ by the Kummer properties 
(\ref{defpolychar}).
Applying it to our case $m=2$, $z=-1$ where 
$\rho_z(\sigma)=\frac{1}{2}(\chi(\sigma)-1)$, 
we obtain
\begin{equation*}
\zeta_{\ell^n}^{\tilchi_2^{z=-1}(\sigma)}=
\sigma\left(
\prod_{a=0}^{\ell^n-1}
(1-\zeta_{2\ell^n}^{2\chi(\sigma)^{-1}a+1})^{\frac{a}{\ell^n}}\right)
\Bigm/
\prod_{a=0}^{\ell^n-1}
(1-\zeta_{2\ell^n}^{2a+\chi(\sigma)})^{\frac{a}{\ell^n}}
\tag{$\ast$}.
\end{equation*}
We evaluate both the denominator and numerator of 
the above right hand side, first by pairing two factors
indexed by $a$ and $a'=-\chi(\sigma)-a$ and by simplifying
their product by
\begin{equation*}
\left(
1-\zeta_{2\ell^n}^{-2a-\chi(\sigma)}
\right)^{\frac{1}{\ell^n}}
=
\left(
1-\zeta_{2\ell^n}^{2a+\chi(\sigma)}
\right)^{\frac{1}{\ell^n}}
\cdot
\zeta_{2\ell^n}^{l^n-\la 2a+\chi(\sigma) \ra}
\end{equation*}
with $0\le \la 2a+\chi(\sigma) \ra \le 2\ell^n$ 
being the unique residue of 
$2a+\chi(\sigma)$ mod $2\ell^n$.
Pick a disjoint decomposition of the index set 
$S:=\{0\le a\le \ell^n-1\}$ into
$S_+\cup S_-\cup S_0$ so that,
for all $a\in S$,

(i) $a\in S_+$ iff $\la -\chi(\sigma)-a\ra\in S_-$;

(ii) $a\in S_0$ iff $a\equiv -a-\chi(\sigma)$ mod $\ell^n$.

\noindent
Then, one finds:
\begin{align*}
\prod_{a\in S-S_0}
(1-\zeta_{2\ell^n}^{2\chi(\sigma)^{-1}a+1})
^{\frac{a}{\ell^n}}
&=
\prod_{a\in S_\pm}
(1-\zeta_{2\ell^n}^{2\chi(\sigma)^{-1}a+1})
^{\frac{-\chi(\sigma)}{\ell^n}}
\zeta_{2\ell^n}
^{(\ell^n-\la 1+2\chi(\sigma)^{-1}a\ra)(-a-\chi(\sigma))}, 
\\
\prod_{a\in S-S_0}
(1-\zeta_{2\ell^n}^{2a+\chi(\sigma)})
^{\frac{a}{\ell^n}}
&=
\prod_{a\in S_\pm}
(1-\zeta_{2\ell^n}^{2a+\chi(\sigma)})
^{\frac{-\chi(\sigma)}{\ell^n}}
\zeta_{2\ell^n}
^{(\ell^n-\la 2a+\chi(\sigma))(-a-\chi(\sigma))}.
\end{align*}
Noting that $\prod_{a\in S}(1-\zeta_{2\ell^n}^{2a+1})=2$,
we obtain the squared sides of $(\ast)$ as
\begin{equation*}
\zeta_{\ell^n}^{2\tilchi_2^{z=-1}(\sigma)}
=\frac{
\sigma\Bigl(
2^{\frac{-\chi(\sigma)}{\ell^n}}
\prod_{a\in S} 
\zeta_{2\ell^n}
^{(\ell^n-\la 1+2\chi(\sigma)^{-1}a\ra)(-a-\chi(\sigma))}
\Bigr)}
{2^{\frac{-\chi(\sigma)}{\ell^n}}
\prod_{a\in S} 
\zeta_{2\ell^n}
^{(\ell^n-\la 2a+\chi(\sigma))(-a-\chi(\sigma))}
}.
\end{equation*}
Here, note that contribution from $S_0$ 
(which is empty when $\ell=2$) is included 
into the factor
$2^{\frac{-\chi(\sigma)}{\ell^n}}$ 
both in the numerator and the denominator.
Now, choose integers $c,\bar{c}\in\Z$ so that
$c\equiv \chi(\sigma)$, $c \bar{c}\equiv 1$
mod $2\ell^n$. Then,
we obtain the following congruence equation
mod $\ell^n$:
\begin{align*}
2\tilchi_2^{z=-1}(\sigma)
&\equiv
-\chi(\sigma)\rho_2(\sigma)
+
\frac{1}{2}\sum_{a\in S}
\chi(\sigma)(-a-c)(\ell^n-\la 1+\bar{c}a\ra)
-(-a-c)(\ell^n-\la 1+2\bar{c}a\ra) \\
&\equiv
-\chi(\sigma)\rho_2(\sigma)
+
\frac{1}{2}\sum_{a\in S}
(-a-c)\left[
\frac{\chi(\sigma)-1}{2}
+\left\{\frac{2a+c}{2\ell^n}\right\}
-c\left\{\frac{1+2\bar{c}a}{2\ell^n}\right\}
\right] \\
&\equiv
-\chi(\sigma)\rho_2(\sigma)
+
\frac{1}{2}\sum_{b\in S}
b\left[
c\left\{\frac{1+2\bar{c}b}{2\ell^n}\right\}
-\left\{\frac{c+2b}{2\ell^n}\right\}
+\frac{1-c}{2}
\right].
\end{align*}
By basic properties of the Bernoulli polynomial
$B_2(X)=X^2-X+\frac{1}{6}$ (cf.~\cite{La}),
the last sum is congruent modulo
$\frac{\ell^n}{48}\Z$ to
\begin{align*}
&\sum_{b\in S}\frac{\ell^n}{2}
\left[
c^2 B_2\left
(\left\{\frac{1+2\bar{c}b}{2\ell^n}\right\}\right)
-B_2\left(\left\{\frac{2b+c}{2\ell^n}\right\}\right)
\right] \\
&=\frac{1}{2}(\chi(\sigma)^2-1)
B_2(\frac{1}{2})
=-\frac{1}{24}{\chi(\sigma)^2-1}.
\end{align*}
Summing up, we find the congruence relations
$$
2\tilchi_2^{z=-1}(\sigma)
\equiv
-\chi(\sigma)\rho_2(\sigma)
-\frac{1}{24}(\chi(\sigma)^2-1)
\quad \text{ mod }\frac{\ell^n}{48}\Z
$$
for all $n$, hence the equality in $\Z_\ell$. 
This concludes the proof of the corollary.
\qed

\medskip
Turning to Theorem \ref{MainFormulaInhomo},
by specialization to 
the case $n=2$ (but for general $k$), we obtain:

\begin{Corollary}
Along the paths from $\ovec{01}$ to $\pm z, z^2$ 
on $\mathbf{P}^1-\{0,1,\infty\}$
used in
Corollary \ref{coro7.2}, it holds that
$$
\tilchi_k^{z^2}(\sigma)=
2^{k-1}\tilchi_k^z(\sigma)
+
\sum_{d=1}^k\binom{k-1}{d-1}2^{d-1}
\chi(\sigma)^{k-d}\tilchi_d^{-z}(\sigma)
\qquad
(k\ge 1,\ \sigma\in G_K).
$$
\end{Corollary}

\noindent
Upon observing special cases of the above formula, 
we find that $\tilchi_4^{z=-1}$ does not factor through
$\Gal(\Q(\mu_{\ell^\infty})/\Q)$, because it involves
a nontrivial term from $\tilchi_3^{\ovec{10}}(\sigma)$ which does
not vanish on $\Gal(\bQ/\Q(\mu_{\ell^\infty}))$
by Soul\'e [So].

\medskip
With regard to the classical formula 
$Li_{2k}(-1)=(-1)^{k+1}(1-2^{2k-1})B_{2k}\frac{\pi^{2k}}{(2k)!}$ ([Le]), 
we should rather figure out its $\ell$-adic analog 
in terms of
the ``$\Q_\ell$-adic''
polylogarithmic characters introduced in Definition \ref{defofQellPolylog}.
In fact,

\begin{Corollary}
Let $\gamma_{-1}:\ovec{01}\nyoroto (z\!=\!-1)$ be the path
in Corollary \ref{dilogspecial}. Then, along
the $\Q_\ell$-path
$x^{-\frac{1}{2}}\gamma_{-1}: \ovec{01}\nyoroto (z\!=\!-1)$,
it holds that
\begin{equation*}
\tilbchi_{2k}^{z=-1}(\sigma)=
\frac{(1-2^{2k-1})}{2^{2k}}\frac{B_{2k}}{2k}
(\chi(\sigma)^{2k}-1)
\qquad (\sigma\in G_\Q).
\end{equation*}
\end{Corollary}

\begin{proof}
{}
Applying Corollary \ref{tilchi-homogeneous}  
to the case where $n=2$ and $p:\ovec{01}\nyoroto \ovec{10}$ is the
straight path on $V_2=\mathbf{P}^1-\{0,\pm 1,\infty\}$, we obtain
$$
\tilbchi_{k}^{\ovec{10},\pi_2(p)}(\sigma)=
2^{k-1}(\tilbchi_{k}^{\ovec{10},\mathj_1(p)}(\sigma)+\tilbchi_{k}^{z=-1,\gamma_{-1}}(\sigma)).
$$
Since $\pi_2(p)$ and $\mathj_1(p)$ are the same standard path 
$\ovec{01}\nyoroto \ovec{10}$ on $V_1$, the values 
$\tilbchi_{k}^{\ovec{10},\pi_2(p)}(\sigma)$ and $\tilbchi_{k}^{\ovec{10},\mathj_1(p)}(\sigma)$
coincide with the (extended)  Soul\'e value $\tilchi_{k}^{\ovec{10}}(\sigma)$
(cf.\,\cite[Remark 2]{NW1}).
The desired formula follows then from 
a basic formula from \cite[Proposition 5.13]{NW2}:
$\tilchi_{2k}^{\ovec{10}}(\sigma)=
\frac{B_{2k}}{2(2k)}(\chi(\sigma)^{2k}-1)$
$(\sigma\in G_\Q)$.
\end{proof}

Unlike the $\Z_\ell$-integral 
analog stated in Corollary \ref{dilogspecial},
the above right hand side generally 
has denominators in $\Q_\ell$.
This is due to the concern of $x^{-\frac{1}{2}}\in\pi_{\Q_l}(\ovec{01})$ 
which does not lie in $\pi_1^{\ell}(V_1\otimes\bK,\ovec{01})$
when $\ell=2$.

\bigskip
{\it Acknowledgement}: 
This work was partially supported by JSPS KAKENHI Grant Number JP26287006.

\ifx\undefined\bysame
\newcommand{\bysame}{\leavevmode\hbox to3em{\hrulefill}\,}
\fi

\end{document}